\documentclass[11pt, amsfonts]{amsart}

\usepackage{amsmath,amssymb,amsthm,enumerate,mathrsfs}
\usepackage[all]{xy}
\usepackage{color}
\usepackage{graphicx}
\usepackage{comment}

\SelectTips{eu}{12}

\newtheorem{thm}{Theorem}[section]
\newtheorem{prop}[thm]{Proposition}
\newtheorem{lem}[thm]{Lemma}
\newtheorem{cor}[thm]{Corollary}

\theoremstyle{definition}
\newtheorem{dfn}[thm]{Definition}
\newtheorem{eg}[thm]{Example}

\theoremstyle{remark}
\newtheorem{rem}[thm]{Remark}

\newcommand{\ZZ}{\mathbb{Z}}
\newcommand{\RR}{\mathbb{R}}

\newcommand{\cl}{\mathrm{cl}}
\newcommand{\scl}{\mathrm{scl}}
\renewcommand{\fill}{\mathrm{fill}}

\newcommand{\Ham}{\mathrm{Ham}}
\newcommand{\Hom}{\mathrm{Hom}}
\newcommand{\Symp}{\mathrm{Symp}}
\newcommand{\Aut}{\mathrm{Aut}}

\newcommand{\qad}{\text{$G$-quasi}}
\newcommand{\ad}{G}
\newcommand{\bg}{h}
\newcommand{\hg}{g}
\newcommand{\bG}{N}
\newcommand{\hG}{G}
\newcommand{\fk}{f}
\newcommand{\gk}{g}
\newcommand{\ppi}{q}
\newcommand{\QQQ}{\mathrm{Q}}
\newcommand{\HH}{\mathrm{H}}
\newcommand{\SSS}{\mathscr{S}}

\title{Bavard's duality theorem for mixed commutator length}

\author[M. Kawasaki]{Morimichi Kawasaki}
\address[Morimichi Kawasaki]{Department of Mathematical Sciences, College of Science and Engineering,
Aoyama Gakuin University, 5-10-1 Fuchinobe, Chuo-ku, Sagamihara-shi, Kanagawa, 252-5258, Japan}
\email{kawasaki@math.aoyama.ac.jp}

\author[M. Kimura]{Mitsuaki Kimura}
\address[Mitsuaki Kimura]{Department of Mathematics, Kyoto University, Kitashirakawa Oiwake-cho, Sakyo-ku, Kyoto 606-8502, Japan}
\email{mkimura@math.kyoto-u.ac.jp}

\author[T. Matsushita]{Takahiro Matsushita}
\address[Takahiro Matsushita]{Department of Mathematical Sciences, University of the Ryukyus, Nishihara-cho, Okinawa 903-0213, Japan}
\email{mtst@sci.u-ryukyu.ac.jp}

\author[M. Mimura]{Masato Mimura}
\address[Masato Mimura]{Mathematical Institute, Tohoku University, 6-3, Aramaki Aza-Aoba, Aoba-ku, Sendai 9808578, Japan}
\email{m.masato.mimura.m@tohoku.ac.jp}



\begin{document}

\baselineskip.525cm

\maketitle

\begin{abstract}
Let $\bG$ be a normal subgroup of a group $\hG$. A quasimorphism $f$ on $\bG$ is $\hG$-invariant if $f(gxg^{-1}) = f(x)$ for every $g \in \hG$ and every $x \in \bG$. The goal 
 in this paper is to establish
Bavard's duality theorem of $\hG$-invariant quasimorphisms, which was previously proved by Kawasaki and Kimura in the case $\bG = [\hG,\bG]$.

Our duality theorem provides
a connection between $G$-invariant quasimorphisms and $(\hG,\bG)$-commutator lengths. Here for $x \in [\hG,\bG]$, the $(\hG,\bG)$-commutator length $\cl_{\hG,\bG}(x)$ of $x$ is the minimum number $n$ such that $x$ is a product of $n$ commutators which are written as $[\hg,\bg]$ with $g \in \hG$ and $h \in \bG$.
In the proof, we give a geometric interpretation of $(\hG,\bG)$-commutator lengths.

As an application of our Bavard duality, we obtain a sufficient condition on a pair $(\hG,\bG)$ under which $\scl_\hG$ and $\scl_{\hG,\bG}$ are bi-Lipschitzly equivalent on $[\hG,\bG]$.

\end{abstract}

\section{Introduction}

\subsection{$G$-invariant quasimorphisms}
A real-valued function $f \colon \hG \to \RR$ on a group $\hG$ is called a \textit{quasimorphism} if
there exists a non-negative number $D$ satisfying
\[|f(\gk_1 \gk_2) - f(\gk_1) - f(\gk_2)| \le D\]
for every pair $\gk_1$ and $\gk_2$ of elements in $\hG$. Such a smallest $D$ is called the \textit{defect} of $f$ and denoted by $D(f)$. More precisely,
\[ D(f) = \sup_{\gk_1, \gk_2 \in \hG}|f(\gk_1 \gk_2) - f(\gk_1) - f(\gk_2)|.  \]
A quasimorphism $f$ on $\hG$ is said to be \textit{homogeneous} if $f(x^n) = n \cdot f(x)$ for every $x \in \hG$ and every integer $n$. Quasimorphisms have been extensively studied in geometric group theory, and are closely related to the second bounded cohomology of groups. For an introduction to this subject, we refer to \cite{Ca} and \cite{Fr}.

Let $\bG$ be a subgroup of $\hG$, and consider a quasimorphism $f$ on $\bG$. It is quite natural to ask when $f$ extends to a quasimorphism on the whole group $\hG$. Such a problem has been actually studied in Shtern \cite{Sh}, Kawasaki \cite{Ka18} and Kawasaki--Kimura \cite{KK}.

In this paper, we treat the case where
$\bG$ is normal.
In this case, there is a condition that any extendable quasimorphism on $\bG$ clearly satisfies.
More precisely,
if a homogeneous quasimorphism $f$ on $\bG$ is extendable, then $f$ is invariant by the
action of $\hG$ on $\bG$ by conjugations, {\it i.e.}, $f(gxg^{-1}) = f(x)$ for every $g \in G$ and $x \in N$.
We call such a quasimorphism on $\bG$ a \emph{$G$-invariant} quasimorphism. 
However, it is known that there exists a $G$-invariant quasimorphism which is not extendable (see \cite{KK}).
Hence, the next natural problem may be to ask when a $G$-invariant homogeneous quasimorphism on $\bG$ is extendable.
To address this problem, Kawasaki and Kimura \cite{KK}
came up with a generalization of Bavard's duality theorem
related to $G$-invariant homogeneous quasimorphisms, which is explained in the next subsection.

Certain types of $\ad$-invariant quasimorphisms have appeared in the literature \cite{BM}.
Let $\bG$ and $\Gamma$ be groups and $\varphi \colon \Gamma \to \Aut(\bG)$ an action of $\Gamma$ on $\bG$. A quasimorphism $f \colon \bG \to \RR$ is $\Gamma$-invariant if $f(\varphi(\gamma) x) =f(x)$ for every $\gamma \in \Gamma$ and for every $x \in \bG$. Set $\hG = \bG \rtimes_\varphi \Gamma$, the semi-direct product.
Then a homogeneous quasimorphism $f$ on $\bG$ is $\Gamma$-invariant if and only if it is $\ad$-invariant.

We also note that $\hG$-invariant quasimorphisms often appear in symplectic geometry. 
Recall that any symplectic manifold $(M,\omega)$ has the following two natural transformation groups. One is the group $\Symp(M,\omega)$ of symplectomorphisms, and the other is the group $\Ham(M,\omega)$ of Hamiltonian diffeomorphisms. It is known that there exist various $\Symp(M,\omega)$-invariant quasimorphisms on $\Ham(M,\omega)$ (see \cite{EP03}, \cite{Py06}, and \cite{Shelukhin} for example).
Moreover, there are applications of $\hG$-invariant quasimorphisms and their extendability in symplectic geometry.
For example, the non-extendability of a certain quasimorphism on $\Ham(M,\omega)$
provides a restriction
on commuting elements of $\Symp_0(M,\omega)$ (see \cite{KK} and \cite[Theorem~1.1]{KKMM}).

\subsection{$(\hG,\bG)$-commutator length and its stabilization}

The \textit{commutator length $\cl_\hG(x)$ of an element $x$ in $[\hG,\hG]$} is the minimum $n$ such that there exist $n$ commutators $c_1, \cdots, c_n$ of $\hG$ with $x = c_1 \cdots c_n$. Then it is known that the limit
\[\scl_\hG(x) = \lim_{n \to \infty} \frac{\cl_\hG(x^n)}{n}\]
exists; we call $\scl_\hG(x)$ the \textit{stable commutator length of $x$}. Bavard's duality theorem gives a deep connection between quasimorphisms and commutator lengths in the following form:

\begin{thm}[Bavard \cite{Bavard}] \label{thm Bavard}
Let $\hG$ be a group and let $x \in [\hG,\hG]$. Then the following equality holds:
\[\scl_\hG(x) = \frac{1}{2} \sup_{f \in \QQQ^{\mathrm{h}}(\hG) - \HH^1(\hG)} \frac{|f(x)|}{D(f)}.\]
\end{thm}

Here $\QQQ^{\mathrm{h}}(\hG)$ is the set of homogeneous quasimorphisms on $\hG$, and $\HH^1(\hG)$ is the set of homomorphisms from $\hG$ to $\RR$. We
regard the right of the equality in Theorem \ref{thm Bavard}
as $0$ if every homogeneous quasimorphism on $\hG$ is a homomorphism.
Some variants of Theorem \ref{thm Bavard} were discussed in \cite{CZ} and \cite{Ka17}.

Theorem \ref{thm Bavard} is also related to bounded cohomology.
It is known that the kernel of the comparison map $\HH_b^2(\hG) \to \HH^2(\hG)$, from the second bounded cohomology to the second ordinary cohomology, is isomorphic to $\QQQ^{\mathrm{h}}(\hG)/\HH^1(\hG)$ (see \cite{Ca}). Thus, Theorem \ref{thm Bavard} gives a necessary and sufficient condition for the comparison map to be injective, \textit{i.e.}, the comparison map is injective if and only if the stable commutator length is identically zero.

Theorem \ref{thm Bavard} has several interesting applications. For example, Endo and Kotschik \cite{EK} used Theorem \ref{thm Bavard} to show the existence of homogeneous quasimorphisms which are not homomorphisms on the mapping class groups of surfaces. Other main applications of Theorem \ref{thm Bavard} are computations of stable commutator lengths: 
Theorem \ref{thm Bavard} allows us to compute the stable commutator lengths when $\hG$ 
admits only few homogeneous quasimorphisms which are not homomorphisms (see \cite{GS}, \cite{Maruyama}, and \cite{Zhuang}). Yet other
applications of Theorem \ref{thm Bavard} are found in \cite{BIP}, \cite{CMS}, and \cite{Mimura} for example.

The goal in this paper is to establish Bavard's duality theorem in the setting of $G$-invariant homogeneous quasimorphisms.
The counterpart for the commutator length in our setting is the $(\hG,\bG)$-commutator length, defined as follows:
an element $x$ in $\hG$ is a \textit{$($single$)$ $(\hG,\bG)$-commutator} if there exist $g \in \hG$ and $h \in \bG$ such that $x = [\hg,\bg]$. As is usual, we denote by $[\hG,\bG]$ the subgroup of $\hG$ generated by the $(\hG,\bG)$-commutators. In some literature, $[\hG,\bG]$ is called the \emph{mixed commutator subgroup.}
The \textit{$(\hG,\bG)$-commutator length $\cl_{\hG,\bG}(x)$ of an element $x$ in $[\hG,\bG]$} is the minimal number $n$ such that there exist $(\hG,\bG)$-commutators $c_1, \cdots, c_n$ with $x = c_1 \cdots c_n$. Then it is clear that there exists a limit
\[\scl_{\hG,\bG}(x) = \lim_{n \to \infty} \frac{\cl_{\hG,\bG}(x^n)}{n},\]
and call $\scl_{\hG,\bG}(x)$ the \textit{stable $(\hG,\bG)$-commutator length of $x$}.

To state our main result, we prepare some definitions:
let $\QQQ^{\mathrm{h}}(\bG)^{\hG}$ be the set of homogeneous $\qad$-invariant quasimorphisms on $\bG$. Let $\HH^1(\bG)^{\hG}$ be the set of homomorphisms from $\bG$ to $\RR$, which are $G$-invariant.
Then, our main result is formulated as follows:

\begin{thm}[Bavard's duality theorem for $\hG$-invariant quasimorhisms, Theorem \ref{thm 7.1}] \label{main thm}
Let $\hG$ be a group and $\bG$ a normal subgroup of $\hG$, and let $x \in [\hG,\bG]$. Then the following equality holds:
\[\scl_{\hG,\bG}(x) = \frac{1}{2}  \sup_{f \in \QQQ^{\mathrm{h}}(\bG)^{\hG} - \HH^1(\bG)^{\hG}} \frac{|f(x)|}{D(f)}.\]
\end{thm}

Note that Theorem \ref{thm Bavard} is the case $G = \bG$ of Theorem \ref{main thm}. Therefore, Theorem \ref{main thm} is a generalization of Bavard's duality theorem.
Kawasaki and Kimura \cite{KK} proved Theorem \ref{main thm} under the extra assumption that $[\hG,\bG] = \bG$.

Our proof of Theorem \ref{main thm} is a generalization of the original proof of Bavard \cite{Bavard}. However, in the proof of Theorem \ref{main thm}, we introduce several notions which did not appear in the original proof. One of the important by-products is a geometric interpretation of $(\hG,\bG)$-commutator lengths (Theorem \ref{main thm 2}).

\subsection{Geometric interpretation of $(\hG,\bG)$-commutator lengths}

Let $x$ be an element of the commutator subgroup $[\hG,\hG]$ of $\hG$. Then $x$ is identified with a homotopy class of loops in the classifying space $B{\hG}$ of $\hG$. Since $x$ vanishes in $\HH_1(G ;\ZZ)$, there exist an oriented compact surface $S$ with connected boundary and a continuous map $f \colon S \to BG$ such that the homotopy class of the loop $f|_{\partial S} \colon \partial S \to BG$ is $x$. Then the commutator length of $x$ coincides with the minimum genus of such a surface $S$ (see \cite{Ca}).

In the proof of Theorem \ref{main thm}, we need a similar geometric 
interpretation of $(\hG,\bG)$-commutator lengths. To obtain this, we introduce $(\hG,\bG)$-simplicial surfaces as follows.

Throughout the paper, every surface is assumed to be compact and oriented. We assume that our triangulation of a surface satisfies the following conditions:
\begin{itemize}
\item Every edge has an orientation; the endpoints of an edge may coincide.

\item Every triangle ($2$-cell) $\sigma$ is surrounded by three edges $\partial_0 \sigma, \partial_1 \sigma$, and $\partial_2 \sigma$ as is depicted in Figure \ref{fig 1}. We do not assume that $\partial_0 \sigma$, $\partial_1 \sigma$, and $\partial_2 \sigma$ are distinct.
\end{itemize}
Let $\SSS_n$ be the set of $n$-simplices of $S$.
A \textit{$\hG$-labelling of $S$} is a function $f \colon \SSS_1 \to \hG$ satisfying $f(\partial_1 \sigma) = f(\partial_2 \sigma) \cdot f(\partial_0 \sigma)$ for every $\sigma \in \SSS_2$.
We call a triple $(S,\SSS, f)$ consisting of a triangulated surface $(S,\SSS)$ together with a $\hG$-labelling $f$ a \textit{$\hG$-simplicial surface}. A \textit{$\hG$-simplicial surface with boundary $x$} is a $\hG$-simplicial surface $(S,\SSS, f)$ such that $\partial S$ has only one edge and $f$ sends it to $x$.

Note that if a $\hG$-labelling of $(S,\SSS)$ is given, then there exists a continuous map $f' \colon S \to BG$ sending $e \in \SSS_1$ to the loop whose homotopy class is $f(e)$. Note that if a $\hG$-labelling of $(S,\SSS)$ is given, then there exists a continuous map $f' \colon S \to BG$ sending $e \in \SSS_1$ to the loop associated to $f(e)$.
Conversely, if a continuous map $f' \colon S \to BG$ is given, then there exists a $\hG$-labelling $f$ of $S$ such that the associated continuous map of $f$ is homotopy equivalent to $f'$. Therefore for $x \in [\hG,\hG]$, the commutator length of $x$ coincides with the minimum genus of a $\hG$-simplicial surface with boundary $x$.

\begin{figure}[t]
\begin{picture}(100,100)(0,0)
\put(0,40){\vector(3,-2){60}}
\put(0,40){\vector(3,1){96}}
\put(60,0){\vector(1,2){36}}

\put(48,32){$\sigma$}

\put(80,28){$\partial_0 \sigma$}
\put(18,8){$\partial_2 \sigma$}
\put(32,60){$\partial_1 \sigma$}

\end{picture}
\caption{} \label{fig 1}
\end{figure}

Now we are ready to define $(\hG,\bG)$-simplicial surfaces.
A \textit{$(\hG,\bG)$-labelling of $S$} is a $\hG$-labelling $f \colon \SSS_1 \to \hG$ such that either $f(\partial_0 \sigma)$ or $f(\partial_2 \sigma)$ belongs to $\bG$ for every $\sigma \in \SSS_2$.
We call a  triple $(S,\SSS, f)$ consisting of triangulated surface $(S,\SSS)$ together with a $(\hG,\bG)$-labelling $f$ a \textit{$(\hG,\bG)$-simplicial surface}.

It turns out that there is a close relation between $(\hG,\bG)$-simplicial surfaces and $(\hG,\bG)$-commutators.
More precisely, we show that for an element $x$ of $\hG$ there exists a $(\hG,\bG)$-simplicial surface with boundary $x$ if and only if $x$ is contained in $[\hG,\bG]$ (see Section \ref{section Geometric characterization}).
Moreover, $(\hG,\bG)$-simplicial surfaces give the following
interpretation of $(\hG,\bG)$-commutator lengths.

\begin{thm}[Geometric interpretation of the $(\hG,\bG)$-commutator length, Theorem \ref{thm 6.2}] \label{main thm 2}
Let $\hG$ be a group and $\bG$ a normal subgroup of $\hG$. For an element $x$ in $[\hG,\bG]$, the $(\hG,\bG)$-commutator length of $x$ coincides with the minimum of the genus of a connected $(\hG,\bG)$-simplicial surface with boundary $x$.
\end{thm}


\subsection{$\bG$-quasimorphisms and filling norms}

The defect $D \colon \QQQ(\hG) \to \RR$ is a seminorm on the space $\QQQ(\hG)$ of quasimorphisms, and the kernel of $D$ is the space $\HH^1(\hG)$ of homomorphisms from $\hG$ to $\RR$. Thus the space $\QQQ(\hG) / \HH^1(\hG)$ is a normed space equipped with the norm induced by $D$.

Let $C_n(\hG)$ be the inhomogeneous complex of the group $\hG$. Namely, $C_n(\hG)$ is the free $\RR$-module generated by $G^n$, and the differential $\partial \colon C_n(\hG) \to C_{n-1}(\hG)$ is defined by
\[\partial(\gk_1, \cdots, \gk_n) = (\gk_2, \cdots, \gk_n) + \sum_{i=1}^{n-1} (-1)^i (\gk_1, \cdots, \gk_i \gk_{i+1}, \cdots, \gk_n) + (-1)^n (\gk_1, \cdots, \gk_{n-1}).\]
Regard $C_n(\hG)$ as the normed space endowed with the $\ell^1$-norm. Then the space of cycles $Z_n(\hG)$ is a closed subspace and hence $C_n(\hG) / Z_n(\hG)$ is a normed space. An important observation in the proof of the original Bavard duality is to identify the normed space $\QQQ(\hG) / \HH^1(\hG)$ with the
continuous dual of $C_2(\hG) / Z_2(\hG)$; then the Hahn--Banach theorem was applied.


In the case of $G$-invariant quasimorphisms, we instead consider  the space of $\bG$-quasimorphisms \cite{EP06}, defined as follows.
A function $f \colon \hG \to \RR$ is an \textit{$\bG$-quasimorphism} if there exists a non-negative number $D''$ such that $|f(\gk x) - f(\gk ) - f(x)| \le D''$ and $|f(xg) - f(x) - f(\gk )| \le D''$ for every $g \in \hG$ and every $x \in \bG$. We call the minimum of such $D''$ the \textit{defect} of the $\bG$-quasimorphism $f$, and denote it by $D''(f)$. We call an $\bG$-quasimorphism with $D''(f) = 0$ an \textit{$\bG$-homomorphism}.

It turns out that $\bG$-quasimorphisms are closely related to $G$-invariant quasimorphisms.
More precisely, a homogeneous quasimorphism $f$ on $\bG$ is $G$-invariant if and only if $f$ extends to $\hG$ as an $\bG$-quasimorphism (see Section 2).

As is the case of the defect of usual quasimorphisms, the defect $D''$ of $\bG$-quasimorphisms is a seminorm of the space $\QQQ_\bG(\hG)$ of $\bG$-quasimorphisms, and the kernel of $D''$ is the space $\HH^1_\bG(\hG)$ of $\bG$-homomorphisms. Then the normed space $\QQQ_\bG(\hG) / \HH^1_\bG(\hG)$ is identified with the
continuous dual of $C_2'(\hG) / Z_2'(\hG)$.
Here $C_2'(\hG)$ is the submodule of $C_2(\hG)$ generated by the set of elements $(\gk_1, \gk_2) \in \hG \times \hG$ such that either $\gk_1$ or $\gk_2$ belongs to $\bG$, and $Z_2'(\hG)$ is $C_2'(\hG) \cap Z_2(\hG)$.

Let $B_1'$ be the image of $\partial \colon C_2' \to C_1$. We 
regard $B_1'$
as a normed space whose norm is given by the isomorphism $C_2' / Z_2' \xrightarrow{\cong} B_1'$.
It turns out that $x \in [\hG,\bG]$ implies that $x \in B_1'$, and the following limit exists:
\[\fill_{\hG,\bG}(x) = \lim_{n \to \infty} \frac{\| x^n\|}{n}.\]
We call $\fill_{\hG,\bG}(x)$ the \textit{$(\hG,\bG)$-filling norm of $x$}.
Using the geometric
interpretation of the $(\hG,\bG)$-commutator length, we show the following theorem:

\begin{thm}[Theorem \ref{thm 6.1}] \label{main thm 3}
For each $x \in [\hG,\bG]$,
\[\fill_{\hG,\bG}(x) = 4 \cdot \scl_{\hG,\bG}(x).\]
\end{thm}

By applying the Hahn--Banach theorem to $B_1'$, we deduce our duality theorem (Theorem \ref{main thm}) from this theorem. This is the outline of the proof of Theorem \ref{main thm}.

\subsection{On the equivalence of $\scl_\hG$ and $\scl_{\hG,\bG}$ on $[\hG,\bG]$}
As an application of Theorem \ref{main thm}, we
study conditions under which $\scl_\hG$ and $\scl_{\hG,\bG}$ are (bi-Lipschitzly) equivalent on $[\hG,\bG]$.
Two functions $\varphi, \psi \colon X \to \RR$ on a set $X$ are \emph{bi-Lipschitzly equivalent}, or \textit{equivalent} for short, if there exists $C \ge 1$ such that $C^{-1} \cdot \psi(x) \le \varphi(x) \le C \cdot \psi(x)$ holds for every $x \in X$.
Let $q \colon G \to G/N$ denote the projection.
Kawasaki and Kimura \cite{KK} showed that under certain
conditions, $\scl_{\hG,\bG}$ and $\scl_\hG$ are equivalent on $[\hG,\bG]$.
Using Theorem \ref{main thm}, we can relax the conditions of Kawasaki and Kimura, and obtain the following theorem:

\begin{thm}[Theorem \ref{thm scl equivalence}] \label{scl eq}
Assume that there exists a subgroup $\Lambda$ of finite index of $\hG / \bG$ having a group homomorphism $s \colon \Lambda \to \hG$ such that $\ppi \circ s (x) = x$ for every $x \in \Lambda$.
Then, for every $x \in [\hG,\bG]$, the following inequalities hold:
\[
\scl_\hG(x) \le \scl_{\hG, \bG}(x) \le 2 \cdot \scl_\hG(x).
\]
\end{thm}
The condition in Theorem~\ref{scl eq} is stated as the short exact sequence $1\to \bG\to \hG \to \hG/\bG\to 1$ \emph{virtually splits} (see Definition~\ref{defn=vs}) in some literature.
Theorem~\ref{scl eq} and Proposition~\ref{prop=extension} below provide an obstruction to a virtual splitting.

The following are examples satisfying the assumptions of Theorem \ref{scl eq}:
\begin{enumerate}[(1)]
\item The short exact sequence $1\to \bG\to \hG\to \hG/\bG\to 1$ splits, meaning that, there exists a section homomorphism $s \colon \hG/\bG \to \hG$ of $\ppi$.
\item The group $\hG/\bG$ is finite.

\item The group $\hG/\bG$ is \emph{virtually free}, that means, it has a subgroup of finite index isomorphic to a free group.
\end{enumerate}
Note that (2) is a special case of (3) because the trivial group may be seen as the free group of rank zero.


Theorem~\ref{scl eq} is proved by the combination of Bavard's duality theorem in our setting (Theorem~\ref{main thm}) and the following \emph{extension theorem} of $\hG$-invariant quasimorphisms:

\begin{prop}[Extension theorem for $\hG$-invariant quasimorphisms, Proposition~\ref{prop 7.3}]\label{prop=extension}
Assume that the exact sequence $1\to \bG \to \hG \to \hG/\bG \to 1$ virtually splits.
Then, for every $\hG$-invariant homogeneous quasimorphism $f$ on $\bG$, there exists a quasimorphism $f'$ on $\hG$ such that $f'|_\bG = f$ and $D(f') \le D(f)$, where $D(f)$ is the defect of $f$.
\end{prop}
The resulting quasimorphism $f'$ may not be homogeneous in general; nevertheless, via the homogenization process, more precisely, by constructing a new map $\overline{f'}\colon \hG\to \RR$ by
\[
\overline{f'}(x):=\lim_{n\to \infty}\frac{f'(x^n)}{n}
\]
for every $x\in \hG$, we obtain a \emph{homogeneous} quasimorphism $\overline{f'}$ on $\hG$ such that $\overline{f'}$ extends $f$ and $D(\overline{f'})\leq 2D(f)$.
See \cite[Lemma 2.58]{Ca} for more details.

From Theorem~\ref{scl eq}, it is a natural problem to ask when $\cl_{\hG,\bG}$ and $\cl_\hG$ are equivalent on $[\hG,\bG]$. We prove that this equivalence holds if either of (1) and (2) above is satisfied:

\begin{thm}[Theorem \ref{thm cl equivalence 1} and Theorem \ref{thm cl equivalence 2}]\label{thm=cl}
Assume that either of the following conditions is satisfied:
\begin{enumerate}[$(1)$]
\item The projection $G \to \hG/\bG$ has a section homomorphism $\hG/\bG \to \hG$.

\item $\hG/\bG$ is finite.
\end{enumerate}
Then, there exists $C>0$ such that for every $x \in [\hG,\bG]$, the following inequalities hold:
\[
\cl_\hG(x) \le \cl_{\hG,\bG}(x) \le C \cdot \cl_\hG(x).
\]
Moreover, if $(1)$ above is satisfied, then we can take $C=3$.
\end{thm}

In case (2),  our estimation of $C$ may be much worse than the constant $2$ in Theorem \ref{scl eq}; see Theorem~\ref{thm 8.11} for more details.
Later, we will generalize Theorem \ref{thm=cl} (2) (see Theorem \ref{thm cl equivalence 3}).

We note that Kawasaki and Kimura \cite{KK} provided an example of $(\hG,\bG)$ where
$\cl_\hG$ and $\cl_{\hG,\bG}$ are not equivalent (moreover, $\scl_\hG$ and $\scl_{\hG,\bG}$ are not equivalent) on $[\hG,\bG]$. See Remark \ref{remark non-equivalence example}. In the subsequent paper \cite{KKMMM}, we have further applications of Theorem \ref{main thm} to the equivalence problem of $\scl_G$ and $\scl_{G,N}$.

\subsection{Organization of this paper}

Section 2 is devoted to the study of algebraic properties of $G$-invariant quasimorphisms and $\bG$-quasimorphisms. In Section 3, we study the space of $\bG$-quasimorphisms and $(\hG,\bG)$-filling norms.
In Section \ref{section Geometric characterization}, we introduce $(\hG,\bG)$-simplicial surfaces and give a geometric 
interpretation of $(\hG,\bG)$-commutator lengths (Theorem \ref{main thm 2}) and prove Theorem \ref{main thm 3}. In Section \ref{section Proof of the main theorem}, we complete the proof of Theorem \ref{main thm}.
In Section~\ref{section=equi_scl}, we prove Proposition~\ref{prop=extension} and Theorem~\ref{scl eq}.
Section~\ref{section=equi_cl} is for the proof of Theorem~\ref{thm=cl}.

Throughout the paper, we use the symbol $\hG$ for a group and $\bG$ for a \emph{normal} subgroup of $\hG$.
For a group $G$, $e_G$ denotes the group unit of $G$.

\subsection*{Acknowledgement}
The authors thank Yuki Arano for discussions. The authors also thank Takumi Yokota for pointing out some typos in an earlier draft in this paper. The authors are grateful to the anonymous referee for insightful comments, which improve the present paper. The first author, the third author, and the fourth author are partially supported by JSPS KAKENHI Grant Number JP18J00765, 19K14536, 17H04822, respectively.

\section{$\bG$-quasimorphisms} \label{section H-qm}
Here we study some algebraic properties of $G$-invariant quasimorphisms and $\bG$-quasimorphisms. First, we introduce the following notation. For real numbers $a$ and $b$ and for a non-negative number $D$, we write $a \sim_D b$ to mean $|b-a| \le D$.

Recall that a function $f$ on $\bG$ is \textit{$\qad$-invariant} if there exists $D' \ge 0$ such that
\[f(\hg x\hg^{-1}) \sim_{D'} f(x)\]
holds for every $g \in \hG$ and every $x \in \bG$. For a $\qad$-invariant quasimorphism $f$ on $\bG$, we write $D'(f)$ to indicate the number
\[\sup \{ |f(\hg x\hg^{-1}) - f(x)| \; | \; g \in \hG, x \in \bG\}.\]
Let $\QQQ(\bG)^{\hG}$ denote the set of $\qad$-invariant quasimorphisms on $\bG$.
We say that a real-valued function $f$ on $\hG$ is \textit{$\ad$-invariant} if $f(\hg x\hg^{-1}) = f(x)$ for every $\gk \in \hG$ and every $x \in \bG$.

Let $f$ be a quasimorphism on $\hG$. For each $x \in \hG$, it is known that there exists a limit
\[\overline{f}(x) = \lim_{n \to \infty} \frac{f(x^n)}{n},\]
and call the function $\overline{f} \colon \hG \to \RR$ the \textit{homogenization of $f$}.
It is known that $\overline{f}$ is a homogeneous quasimorphism \cite[Lemma 2.21]{Ca}.
Before stating the next lemma, recall that a homogeneous quasimorphism $f$ on $\hG$ is  conjugation invariant, \textit{i.e.},  $f(\hg x\hg^{-1}) = f(x)$ for every pair $g$ and $x$ of elements in $\hG$. Indeed, observe that $(\hg x\hg^{-1})^n=\hg x^n\hg^{-1}$ for all $\hg,x\in \hG$ and all $n\in \ZZ$ (see also Section 2.2.3 in \cite{Ca}).

\begin{lem} \label{lem 3.1}
Let $f$ be a $\qad$-invariant quasimorphism on $\bG$. Then its homogenization $\overline{f}$ is  $\ad$-invariant. In particular, if $f$ is homogeneous, then $D'(f) = 0$.
\end{lem}
\begin{proof}
It is known that $f(x) \sim_{D(f)} \overline{f}(x)$ for every $x \in \hG$ (see Lemma 2.21 of \cite{Ca}). Let $\gk \in \hG$ and $x \in \bG$. For every positive integer $n$, we have
\[n \overline{f}(\hg x\hg^{-1}) = \overline{f}(\gk x^n \hg^{-1}) \sim_{D(f)} f(\gk x^n \gk^{-1}) \sim_{D'(f)} f(x^n) \sim_{D(f)} \overline{f}(x^n) = n \overline{f}(x).\]
Therefore we have an inequality
\[|\overline{f}(\hg x\hg^{-1}) - \overline{f}(x)| \le \frac{2 D(f) + D'(f)}{n}\]
for every positive integer $n$. This implies that $\overline{f}(\hg x\hg^{-1}) = \overline{f}(x)$.
\end{proof}

Here we introduce the following notion relevant to quasimorphisms.

\begin{dfn}
Let $\hG$ be a group and $\bG$ a normal subgroup of $\hG$. A function $f \colon \hG \to \RR$ is called an \textit{$\bG$-quasimorphism} if there exists a non-negative number $D''$ such that
\[|f(\gk x) - f(\gk ) - f(x)| \le D''\]
and
\[|f(xg) - f(x) - f(\gk )| \le D''\]
for every $g \in \hG$ and every $x \in \bG$.
Let $D''(f)$ denote the infimum of such a non-negative number $D''$, and call it the \textit{defect} of the $\bG$-quasimorphism $f$.
\end{dfn}
A concept similar to an $N$-quasimorphism appeared in \cite{EP06} (see also \cite{Ka16}, \cite{Ki}).

\begin{lem} \label{lem 3.3}
Let $f \colon \hG \to \RR$ be an $\bG$-quasimorphism. Then the restriction $f|_\bG$ of $f$ to $\bG$ is a $\qad$-invariant quasimorphism.
\end{lem}
\begin{proof}
By the definition of $\bG$-quasimorphism, $f|_\bG$ is a quasimorphism on $\bG$ whose defect is not larger than $D''(f)$. For $g \in \hG$ and $x \in \bG$, we have
\[ f(\hg x\hg^{-1}) + f(\gk ) \sim_{D''(f)} f(\gk x) \sim_{D''(f)} f(\gk ) + f(x).\]
This means that
\[f(\hg x\hg^{-1}) - f(x) \sim_{2D''(f)} 0,\]
and hence $f$ is a $\qad$-invariant quasimorphism on $\bG$.
\end{proof}

\begin{prop} \label{prop 3.4.1}
Let $f \colon \bG \to \RR$ be a $\qad$-invariant quasimorphism on $\bG$.
Then there exists an $\bG$-quasimorphism $f' \colon \hG \to \RR$ with $f'|_\bG = f$ and $D''(f) \le D(f) + D'(f)$.
Moreover, if $f$ is homogeneous, we can take $f'$ to satisfy $D''(f') = D(f)$.
\end{prop}
\begin{proof}
  Fix a (set-theoretical) section of the quotient map $\hG\to \hG/\bG$ such that $e_{\hG/\bG}\mapsto e_\hG$ and let $S$ be its image.
Note that the map $S \times \bG \to \hG$, $(s,x) \mapsto sx$ is a bijection.
Let $f'$ be a real-valued function on $\hG$ which satisfies the following properties:
\begin{enumerate}[(1)]
\item $f'(e_G) = 0$. For an element $s \in S - \{ e_G\}$, let $f'(s)$ be an arbitrary real number.

\item For $s \in S$ and $h \in \bG$, define $f'(sh) = f'(s) + f(h)$.
\end{enumerate}
We show that $f'$ is an $\bG$-quasimorphism such that $D''(f) \le D(f) + D'(f)$ and $f'|_\bG = f$. Let $g \in \hG$ and $x \in \bG$. Let $s \in S$ and $h \in \bG$ 
with $sh = g$. Then we have
\[f'(\gk x) = f'(shx) = f'(s) + f'(hx) \sim_{D(f)} f'(s) + f'(h) + f'(x) = f'(sh) + f'(x) = f'(\gk ) + f'(x),\]
and hence
\[|f'(\gk x) - f'(\gk ) - f'(x)| \le D(f).\]
Next set $y = g^{-1} x g$. Then we have
\[f'(xg) - f'(x) - f'(\gk ) = f'(\gk y) - f'(\gk yg^{-1}) - f'(\gk ) \sim_{D'(f)} f'(\gk y) - f'(y) - f'(\gk ) \sim_{D(f)} 0,\]
and hence
\[|f'(xg) - f'(x) - f'(\gk )| \le D(f) + D'(f).\]
Thus we have shown that $f'$ is an $\bG$-quasimorphism satisfying $D''(f) \le D(f) + D'(f)$.

Suppose that $f$ is homogeneous. Then Lemma \ref{lem 3.1} implies that $D'(f) = 0$, and hence we have $D''(f) \le D(f)$.
On the other hand, it is clear that $D(f) \le D''(f)$. This completes the proof.
\end{proof}

Combining Lemma \ref{lem 3.3} and Proposition \ref{prop 3.4.1}, we have the following equivalence: a quasimorphism $f \colon \bG \to \RR$ is $\qad$-invariant if and only if $f$ admits an extension $f' \colon \hG \to \RR$ which is an $\bG$-quasimorphism.


An \textit{$\bG$-homomorphism} is a function $f \colon \hG \to \RR$ satisfying
\[f(\gk x) = f(\gk ) + f(x) = f(x \gk )\]
for every $g \in \hG$ and every $x \in \bG$.

\begin{cor}
A $\qad$-invariant homomorphism $f \colon \bG \to \RR$ has an extension $f' \colon \hG \to \RR$ which is an $\bG$-homomorphism.
\end{cor}
\begin{proof}
Since a homomorphism is homogeneous, it follows from Proposition \ref{prop 3.4.1} that there exists an $\bG$-quasimorphism $f' \colon \hG \to \RR$ with $D''(f') = D(f) = 0$.
\end{proof}

We end this section with
discussions on $(\hG,\bG)$-commutators. Recall that an element $x$ of $\hG$ is a \textit{$($single$)$ $(\hG,\bG)$-commutator} if there exist $g \in \hG$ and $h \in \bG$ such that $x = [\hg,\bg] = \hg\bg\hg^{-1} h^{-1}$. Since
\[[\hg,\bg] = [\hg\bg\hg^{-1}, g^{-1}]
\quad \textrm{and} \quad [\bg,\hg] = [g^{-1}, \hg\bg\hg^{-1}],\]
$x$ is a $(\hG,\bG)$-commutator if and only if there exist $g \in \hG$ and $h \in \bG$ with $x = [h,g]$. As is usual, we write $[\hG,\bG]$ to mean the subgroup of $\hG$ generated by the $(\hG,\bG)$-commutators. Note that $[\hG,\bG]$ is a normal subgroup of $\hG$ and $[\hG,\bG] \subset \bG$.

For an element $x$ in $[\hG,\bG]$, define the \textit{$(\hG,\bG)$-commutator length $\cl_{\hG,\bG}(x)$ of $x$} by
\[\cl_{\hG,\bG}(x) = \min \{ n \; | \; \textrm{There exist $n$ $(\hG,\bG)$-commutators $c_1, \cdots, c_n$ such that $x = c_1 \cdots c_n$}\}.\]
We note that $\cl_{\hG,\bG}$ is a $\hG$-invariant function.

It easily follows from Fekete's lemma that there exists a limit
\[\scl_{\hG,\bG}(x) = \lim_{n \to \infty} \frac{\cl_{\hG,\bG}(x^n)}{n},\]
and we call $\scl_{\hG,\bG}(x)$ the \textit{stable $(\hG,\bG)$-commutator length of $x$}.

\begin{lem} \label{lem 3.6}
For a $\qad$-invariant homogeneous quasimorphism $f \colon \bG \to \RR$, the following equalities hold:
\[D(f) = \sup_{h_1, h_2 \in \bG} |f([h_1, h_2])| = \sup_{g\in \hG, h \in \bG} |f([\hg,\bg])|.\]
In particular, $f([\hg,\bg]) \le D(f)$ holds for every $g \in \hG$ and every $h \in \bG$.
\end{lem}
\begin{proof}
The first equality
\[D(f) = \sup_{h_1, h_2 \in \bG} |f([h_1, h_2])|\]
is known (see Lemma 3.6 of \cite{Bavard} or Lemma 2.24 of \cite{Ca}). The inequality
\[\sup_{h_1, h_2 \in \bG} |f([h_1, h_2])| \le \sup_{g \in \hG, h \in \bG} |f([\hg,\bg])| \]
is obvious. Since $f$ is a homogeneous $\qad$-invariant quasimorphism, $f$ is $\ad$-invariant (Lemma \ref{lem 3.1}) and hence satisfies
\[f([\hg,\bg]) = f(\gk hg^{-1}h^{-1}) \sim_{D(f)} f(\gk hg^{-1}) + f(h^{-1}) = f(h) - f(h) = 0.\]
Thus we have
\[\sup_{g\in \hG, h\in \bG} |f([\hg,\bg])| \le D(f),\]
which completes the proof.
\end{proof}

\section{Filling norms}

Let $\QQQ_\bG = \QQQ_\bG(\hG)$ be the space of $\bG$-quasimorphisms on $\hG$, and $\HH^1_\bG = \HH^1_\bG(\hG)$ the space of $\bG$-homomorphisms on $\hG$. Then the defect $D''$ of $\bG$-quasimorphisms is a seminorm on $\QQQ_\bG$ whose kernel is $\HH^1_\bG$, and hence $\QQQ_\bG / \HH^1_\bG$ is a normed space. The
goal in this section is to identify $\QQQ_\bG / \HH^1_\bG$ with a
continuous dual of a certain normed space arising from the inhomogeneous complex of the group $\hG$.

Let $C_n(\hG)$ denote the free $\RR$-module generated by the $n$-tuple direct product $G^n$. For $i \in \{ 0, 1, \cdots, n \}$, define the map $\partial_i \colon C_{n}(\hG) \to C_{n-1}(\hG)$ by the linear map satisfying
\[
\partial_i (g_1, \cdots, g_n) =
\begin{cases}
(g_2, \cdots, g_n) & (i = 0)\\
(g_1, \cdots, g_{i-1}, g_i g_{i+1}, g_{i+2}, \cdots, g_n) & (i = 1, \cdots, n-1) \\
(g_1, \cdots, g_{n-1}) & (i = n),
\end{cases}
\]
and the differential $\partial \colon C_{n}(\hG) \to C_{n-1}(\hG)$ is the linear map satisfying
\[
\partial( \gk_1, \cdots, \gk_n) = \sum_{i=0}^{n} (-1)^i \partial_i ( \gk_1, \cdots, \gk_n).
\]
Note that in the case $n = 2$, the differential $\partial \colon C_2(G) \to C_1(G)$ is described by
\[
\partial(\gk_1, \gk_2) = \gk_2 - \gk_1 \gk_2 + \gk_1.\]

Let $C_2' = C_2'(\hG)$ be the $\RR$-submodule of $C_2(\hG)$ generated by the set
\[\{ (\gk_1, \gk_2) \in \hG \times G \; | \; \textrm{$\bG$ contains either $\gk_1$ or $\gk_2$}\}.
\]
Set $B_1' = \partial C_2'$ and $Z_2' = Z_2(G ; \RR) \cap C_2'$.
Regard $C_2'$ as a normed space endowed with the $\ell^1$-norm. Then $Z_2'$ is a closed subspace of $C_2'$, and hence $C_2' / Z_2'$ is a normed space. We consider $B_1'$ as a normed space by the isomorphism $C_2' / Z_2' \xrightarrow{\cong} B_1'$, and we write $\| x\|'$ to indicate the norm on $B'_1$.

\begin{lem} \label{lem 5.1}
If $g \in \hG$ and $h \in \bG$, then $[\hg,\bg] \in B_1'$ and $\| [\hg,\bg]\|' \le 3$.
\end{lem}
\begin{proof}
From the equality
\[\partial ([\hg,\bg], hg) - \partial (\gk ,h) + \partial (h,g) = [\hg,\bg],\]
the assertion follows.
\end{proof}

\begin{lem} \label{lem 5.2}
If $x,y \in [\hG,\bG]$, then $\| xy\|' \le \| x\|' + \| y\|' + 1$.
\end{lem}
\begin{proof}
Since $x,y \in [\hG,\bG]$, we have $xy \in [\hG,\bG]$.
Therefore Lemma \ref{lem 5.1} implies that $x,y,xy \in B_1'$.
Since $\partial(x,y) = y - xy + x$, we have
\[\| xy - x - y\|' \le 1.\]
Therefore we have
\[\| xy\|' = \| (x + y) + (xy - x - y) \|' \le \| x + y\|' + \| xy - x - y\|' \le \| x\|' + \| y\|' + 1
,\]
as desired.
\end{proof}

Let $x \in [\hG,\bG]$. Lemma \ref{lem 5.2} implies that
\[\| x^{m+n}\|' + 1 \le (\| x^m\|' + 1) + (\| x^n\|' + 1)\]
for every pair $m$ and $n$ of positive integers. Therefore Fekete's lemma implies that there exists a limit
\[\fill_{\hG,\bG}(x) = \lim_{n \to \infty} \frac{\| x^n\|'}{n}.\]
We call $\fill_{\hG,\bG}(x)$ the \textit{$(\hG,\bG)$-filling norm of $x$}.

\begin{prop} \label{prop 3.4}
Let $x \in [\hG,\bG]$. Then we have an inequality $\| x\|' \le 4 \cdot \cl_{\hG,\bG}(x) - 1$.
\end{prop}
\begin{proof}
Suppose that $\cl_{\hG,\bG}(x) = m$, and let $c_1, \cdots, c_m$ be $(\hG,\bG)$-commutators satisfying $x = c_1 \cdots c_m$. Then Lemmas \ref{lem 5.1} and \ref{lem 5.2} imply that
\[\| x\|' \le \| c_1\| + \cdots + \| c_m\| + (m-1) \le 3m+(m-1) = 4m - 1.\]
This completes the proof.
\end{proof}

\begin{cor} \label{cor 3.5}
If $x \in [\hG,\bG]$, then $\fill_{\hG,\bG}(x) \le 4 \cdot \scl_{\hG,\bG}(x)$.
\end{cor}
\begin{proof}
By Proposition \ref{prop 3.4}, we have an inequality
\[\frac{\| x^n\|'}{n} \le 4 \cdot \frac{\cl_{\hG,\bG}(x^n)}{n} - \frac{1}{n}.\]
By taking the limits, we have $\fill_{\hG,\bG}(x) \le 4 \cdot \scl_{\hG,\bG}(x)$.
\end{proof}

The equality in Theorem \ref{main thm 3} can be decomposed into the following two inequalities for each $x \in [\hG,\bG]$:
\[\fill_{\hG,\bG}(x) \leq 4 \cdot \scl_{\hG,\bG}(x)\quad \textrm{and}\quad \fill_{\hG,\bG}(x) \geq 4 \cdot \scl_{\hG,\bG}(x).\]
By the arguments above, we have obtained one side of these inequalities. We will prove the other side in the next section, using a geometric
interpretation of $(\hG,\bG)$-commutator lengths.

For a normed space $V$, let $V^{\ast}$ denote the
continuous dual of $V$.

\begin{prop}
The normed spaces $\QQQ_\bG / \HH^1_\bG$ and $(C_2' / Z_2')^{\ast}$ are isometrically isomorphic.
\end{prop}
\begin{proof}
Let $f \colon \hG \to \RR$ be an $\bG$-quasimorphism on $\hG$. Then $f$ is identified with an $\RR$-linear map $f \colon C_1 (\hG) \to \RR$. Since $f$ is an $\bG$-quasimorphism, we have
\[|f \circ \partial (\gk ,x)| = |f(x) - f(\gk x) + f(\gk )| \le D''(f),\]
and
\[|f \circ \partial (x,g)| = |f(\gk ) - f(xg) + f(x)| \le D''(f)\]
for every $g \in \hG$ and every $x \in \bG$. This means that $f \circ \partial \colon C_2' \to \RR$ is continuous with respect to the $\ell^1$-norm. Since $f \circ \partial$ vanishes on $Z_2'$, we have that $f \circ \partial$ induces a continuous linear map $\hat{f} \colon C_2' / Z_2' \to \RR$, whose operator norm is $D''(f)$. Thus we have constructed the correspondence from $\QQQ_\bG(\hG)$ to $(C_2' / Z_2')^{\ast}$.
The kernel of this correspondence is the space of $\bG$-homomorphisms. Thus we have constructed a norm preserving linear map $\QQQ_\bG / \HH^1_\bG \to (C_2' / Z_2')^{\ast}$.

Next, we construct the correspondence $(C_2' / Z_2')^{\ast} \to \QQQ_\bG / \HH^1_\bG$. Let $f \colon C_2' / Z_2' \to \RR$ be a bounded operator. Then $f$ is identified with a linear map $\overline{f} \colon B_1' \to \RR$. Since our coefficient is $\RR$, there exists a linear map $f' \colon C_1(\hG) \to \RR$ such that $f'|_{B_1'} = \overline{f}$, which is not necessarily continuous. The function $f' \colon C_1 (\hG) \to \RR$ is identified with a real-valued function $f' \colon \hG \to \RR$ on $\hG$. Let $D''$ be the operator norm of $f \colon C_2' \to \RR$. Then we have
\[| f'(\gk x) - f'(\gk ) - f'(x)| = | f'(\partial (\gk ,x))| = | f(\gk ,x)| \le D''.\]
Similarly, we can show $|f'(xg) - f'(x) - f'(\gk )| \le D''$. This implies that $f'$ is an $\bG$-quasimorphism.

To complete the construction of the correspondence $(C_2' / Z_2')^{\ast} \to \QQQ_\bG / \HH^1_\bG$, let $f''$ be another linear extension of $\overline{B}_1 \xrightarrow{\overline{f}} \RR$ and show that $f' - f''$ is an $\bG$-homomorphism. Since $f'$ and $f''$ coincide on $B_1'$, we have
\[f'(\gk x) -f'(\gk ) - f'(x) = f'(\gk x - g - x) = f''(\gk x - g - x) = f''(\gk x) - f''(\gk ) - f''(x).\]
This means that
\[(f' - f'')(\gk x) = (f' - f'')(\gk ) + (f' - f'')(x).\]
Similarly, we can show that
\[(f' - f'')(xg) = (f' - f'')(x) + (f' - f'')(\gk ),\]
and hence $f' - f''$ is an $\bG$-homomorphism. Thus we have completed the construction of the correspondence $(C_2' / Z_2')^{\ast} \to \QQQ_\bG / \HH^1_\bG$.
Since these correspondences are mutually inverses and the correspondence $\QQQ_\bG / \HH^1_\bG \to (C_2'/ Z_2')^{\ast}$ is an isometry, we complete the proof.
\end{proof}

\begin{cor}
The normed space $\QQQ_\bG / \HH^1_\bG$ is a Banach space.
\end{cor}

Applying the Hahn--Banach theorem, we have the following corollary:

\begin{cor} \label{cor 5.8}
For $x \in [\hG,\bG]$, the following holds:
\[\| x\|' = \sup_{f \in \QQQ_\bG - \HH^1_\bG} \frac{|f(x)|}{D''(f)}.\]
\end{cor}

\section{Geometric interpretation of $\cl_{\hG,\bG}$}\label{section Geometric characterization}

Let $x$ be an element of $[\hG,\hG]$. An element of $\hG$ is identified with a homotopy class of loops in $B{\hG}$. Since $x$ is zero in the first integral homology group of $\hG$, there exists a compact orientable surface $S$ with connected boundary and a map $f \colon S \to B\hG$ which sends the boundary of $S$ to $x$. The commutator length is the minimum genus of such a surface $S$. In this section, we
establish an analog of the interpretation above in the setting of the $(\hG,\bG)$-commutator length; the precise statement goes as follows:

\begin{thm} \label{thm 6.1}
For $x \in [\hG,\bG]$, the equality $\fill_{\hG,\bG}(x) = 4 \cdot \scl_{\hG,\bG}(x)$ holds.
\end{thm}

Recall that we have 
showed the inequality $\fill_{\hG,\bG}(x) \le 4 \cdot \scl_{\hG,\bG}(x)$ (Corollary \ref{cor 3.5}). Hence,
it only remains
to show $\fill_{\hG,\bG}(x) \ge 4 \cdot \scl_{\hG,\bG}(x)$.

\subsection{$(\hG,\bG)$-labellings of simplicial surfaces}

We introduce the concept of $(G,N)$-simplicial surfaces, which provides a geometric
interpretation of the $(G,N)$-commutator length $\cl_{G,N}$. We formulate this concept by using $\Delta$-complex. First, we recall some concepts related to $\Delta$-complex; see \cite{H}.

For a non-negative integer $n$, the \textit{standard $n$-simplex} is the subset
\[
\Delta^n = \Big\{ (t_0, \cdots, t_n) \in \RR^{n+1} \; | \; t_0\geq 0,\ t_1\geq 0,\ldots ,\ t_n\geq 0,\ \sum_{i=0}^n t_i = 1 \Big\}
\]
of $\RR^{n+1}$.
We define the \textit{boundary of $\Delta^n$} by
\[
\partial \Delta^n = \{ (t_0, \cdots, t_n) \in \Delta^n \; | \; \textrm{$t_i = 0$ for some $i$}\},
\]
and set $(\Delta^n)^\circ = \Delta^n - \partial \Delta^n$.

For $i \in \{ 0, 1, \cdots, n\}$, define a map $d_i \colon \Delta^n\to \Delta^{n+1}$ by
\[
d_i(t_0, \cdots, t_n) = (t_0, \cdots, t_{i-1}, 0, t_i , \cdots, t_n).
\]

Let $X$ be a topological space. The \textit{$\Delta$-complex structure on $X$} is the family $\SSS = \{ \SSS_n \}_{n \ge 0}$ which satisfies the following three conditions:
\begin{enumerate}[(1)]
\item $\SSS_n$ is a family of continuous maps $\sigma$ from $\Delta^n$ to $X$ such that $ \sigma |_{(\Delta^n)^\circ} \colon (\Delta^n)^\circ \to X$ is a homeomorphism into $X$. For every point $x$ in $X$, there are unique $n$ and unique $\sigma \in \SSS_n$ such that $x$ is contained in $\sigma((\Delta^n)^\circ)$.

\item Let $\sigma \in \SSS_n$ and let $i \in \{ 0,1 \cdots, n-1\}$. Then $\sigma \circ d_i$ is contained in $\SSS_{n-1}$. We write $\partial_i \sigma$ instead of $\sigma \circ d_i$.

\item $A \subset X$ is open if and only if $\sigma^{-1}(A)$ is open in $\Delta^n$ for every $n \ge 0$ and every $\sigma \in \SSS_n$.
\end{enumerate}

We use the term \textit{simplicial surface} to mean a pair $(S, \SSS)$ consisting of a compact connected orientable surface $S$ together with its $\Delta$-complex structure $\SSS$ on $S$.

For a group $G$, a \textit{$G$-labelling of a simplicial surface $(S,\SSS)$} is a function $f \colon \SSS_1 \to G$ satisfying $f(\partial_1 \sigma) = f(\partial_2 \sigma) \cdot f(\partial_0 \sigma)$ for every $\sigma \in \SSS_2$. A $(G,N)$-labelling is a $G$-labelling $f \colon \SSS_1 \to G$ such that for every $\sigma \in \SSS_2$ either $f(\partial_0 \sigma)$ or $f(\partial_2 \sigma)$ is contained in $N$.


\begin{dfn}[$(G,N)$-simplicial surface]
A \textit{$(G,N)$-simplicial surface} is defined to be a triple $(S, \SSS, f)$ such that $(S, \SSS)$ is a simplicial surface and $f$ is a $(G,N)$-labelling of $(S, \SSS)$.

For an element $x$ of $G$, a \textit{$(G,N)$-simplicial surface with boundary $x$} is a $(G,N)$-simplicial surface $(S, \SSS, f)$ such that $\partial S$ is a subcomplex of $(S, \SSS)$ consisting of one edge and one point, and $f$ sends the only one edge of $\partial S$ to $x$. Note that in this case $\partial S$ is homeomorphic to $S^1$.
\end{dfn}

We call an element $x \in B_1'$ an \textit{integral $(\hG,\bG)$-boundary} if there exists a chain $c \in C_2'$ with integral coefficients such that $\partial c = x$. For an integral $(\hG,\bG)$-boundary $x$, we write $\| x\|'_\ZZ$ to mean the infimum of $\| c\|_1$ such that $c \in C_2'$ is a chain with integral coefficients and $\partial c = x$. Here $\| c\|_1$ denotes the $\ell^1$-norm of $c$.

We call an element $x \in B_1'$ an \textit{integral $(\hG,\bG)$-boundary} if there exists a chain $c \in C_2'$ with integral coefficients such that $\partial c = x$. For an integral $(\hG,\bG)$-boundary $x$, we write $\| x\|'_\ZZ$ to mean the infimum of $\| c\|_1$ such that $c \in C_2'$ is integral and $\partial c = x$.

\begin{lem} \label{lem 6.2}
Let $x \in [\hG, \bG]$. Then there exists a $(\hG, \bG)$-simplicial surface $(S, \SSS, f)$ with boundary $x$ whose number of $2$-simplices of $(S,\SSS)$ coincides with $\| x\|'$.
\end{lem}
\begin{proof}
By the proof of Lemma \ref{lem 5.1}, every $(\hG, \bG)$-commutator is an integral $(\hG, \bG)$-boundary.
Set $m = \| x\|'_\ZZ$. Let $c \in C_2'$ be an integral $2$-chain
\[
c = \sum_{k = 1}^m \varepsilon_k (x_k, y_k), \; \varepsilon_k = \pm 1
\]
such that either $x_k$ or $y_k$ belongs to $N$ and $x = \partial c$. Since $\| c\|_1 = \| x\|'_\ZZ$, we have that $(x_k, y_k) = (x_l, y_l)$ implies $\varepsilon_k = \varepsilon_l$.
For each $i \in \{ 1, \cdots, m\}$, we write $\Delta_i^2$ to mean a copy of the standard $2$-simplex $\Delta^2$.
Define a surface $S$ as a quotient of the direct sum $\coprod_k \Delta^2_k$ as follows.

Let $z \in G$. Define sets $A_z$ and $B_z$ by
\[
A_z = \{ (i,k) \in \{ 0,1,2\} \times \{ 1, \cdots, m\} \; | \; \partial_i \varepsilon_k (x_k, y_k) = z\},
\]
\[B_z = \{ (i,k) \in \{ 0,1,2\} \times \{ 1, \cdots, m\} \; | \; \partial_i \varepsilon_k (x_k, y_k) = -z \}. \]
Suppose that $z \ne x$. Since $\partial c = x$, we have that $\# B_z = \# A_z$. Let $\varphi_z \colon B_z \to A_z$ be a bijection. If $z = x$, then $\partial c = x$ implies $\# B_z + 1 = \# A_z$. In this case, let $\varphi_x \colon B_x \to A_x$ be an injection. Then the identification of $\coprod_k \Delta^2_k$ is defined as follows: suppose that $(i,k) \in B_z$ and $(j,l) = \varphi_z(i,k)$.
Then we identify the edge $\partial_i \Delta^2_k$ of $\Delta^2_k$ and the edge $\partial_j \Delta^2_l$ of $\Delta^2_l$. Let $S$ denote the quotient space of $\coprod_k \Delta^2_k$ with respect to this identification. Then $S$ is a surface with boundary homeomorphic to $S^1$. We show that $S$ has a natural $\Delta$-complex structure and a $(G,N)$-coloring, and is orientable and connected.

An $n$-simplex of $S$ is a composite of an $n$-simplex $\Delta^n \to \coprod_k \Delta^2_k$ followed by the quotient map $\coprod_k \Delta^2_k \to S$. This defines a $\Delta$-complex structure $\SSS$ on $S$.

The $(G,N)$-labelling of $(S, \SSS)$ is defined as follows: let $e \colon \Delta^1 \to S$ be a $1$-simplex of $(S,\SSS)$. Then there exist $i \in \{ 0,1,2\}$ and $\xi \in \{ 1, \cdots, m\}$ such that the composite
\[
\Delta^1 \xrightarrow{d_i} \Delta^2_\xi \hookrightarrow \coprod_k \Delta^2_k \to S\]
coincides with $e$, and define $f(e)$ by $\partial_i (x_{\xi}, y_{\xi})$.
This function $f \colon \SSS_1 \to G$ is well-defined by the definition of the identification. To see that this function is a $(G,N)$-labelling, let $\sigma \in \SSS_2$. Then there exists a unique $\xi \in \{ 1, \cdots, m\}$ such that the composite
\[
\Delta^2_\xi \hookrightarrow \coprod_k \Delta^2_k \to S
\]
coincides with $\sigma$. By the definition of $f$, we have that
\[
f(\partial_0 \sigma) = \partial_0 (x_{\xi}, y_{\xi}) = y_{\xi}, f(\partial_1 \sigma) = \partial_1(x_{\xi}, y_{\xi}) = x_{\xi} y_{\xi}, f(\partial_2 \sigma) = \partial_2 (x_{\xi}, y_{\xi}) = x_{\xi}.
\]
Therefore we have that $f(\partial_1 \sigma) = f(\partial_2 \sigma) \cdot f(\partial_0 \sigma)$. Since one of $x_{\xi}$ and $y_{\xi}$ is contained in $N$, we have that $f$ is a $(G,N)$-labelling.

Now we prove that $S$ is orientable. For each $\xi \in \{ 1, \cdots, m\}$, define $\sigma_\xi$ to be the composite of
\[
\Delta^2_\xi \hookrightarrow \coprod_k \Delta^2_k \to S.
\]
Then
\[
\sum_k \varepsilon_k \sigma_k
\]
is an orientation class of $(S,\partial S)$.

Finally, we show that $S$ is connected.
Suppose that $S$ is not connected.
Then there exists a connected component $S_1$ of $S$ such that $\partial S_1 = \emptyset$.
Set
\[
A = \{ i \in \{ 1, \cdots, m\} \; | \; \textrm{$\sigma_i$ is a simplex of $S_1$}\},
\]
and $B = \{ 1, \cdots, m\} - A$. Here $\sigma_{i}$ is the $2$-simplex defined in the previous paragraph. Then we have
\[
\partial \Big( \sum_{k \in A} \varepsilon_k(x_k, y_k) \Big) = 0,
\]
and hence
\[
x = \partial c = \partial \Big( \sum_k \varepsilon_k (x_k, y_k) \Big) = \partial \Big(\sum_{k \in B} \varepsilon_k (x_k, y_k) \Big).
\]
Since $\# B < m$, we have that $\| x\|'_\ZZ \le \# B < m  = \| x\|'_\ZZ$. This is a contradiction.
%
%
%
%
%
\end{proof}

Conversely, we will show that if there exists a $(\hG,\bG)$-simplicial surface with boundary $x$ then $x$ is contained in $[\hG,\bG]$ (see Proposition \ref{prop 6.6}).

We are now ready to state the geometric 
interpretation of $(\hG,\bG)$-commutator lengths.

\begin{thm}\label{thm 6.2}
Let $x$ be an element of $[\hG,\bG]$.
Then the $(\hG,\bG)$-commutator length ${\cl}_{\hG,\bG}(x)$ of $x$ 
equals the minimum of the genus of a $(\hG,\bG)$-simplicial surface with boundary $x$.
\end{thm}

\begin{rem} \label{rem geometric interpretation of cl_G}
Theorem~\ref{thm 6.2} is a generalization of a usual geometric interpretation of the commutator length $\cl_G$ \cite{Ca}.
Indeed, if $(S, \SSS, f)$ is a $(G,G)$-simplicial surface with boundary $x$, then there exists a continuous map $f' \colon S \to BG$ defined as follows. For each $0$-simplex $v$ of $S$, define $f(v)$ to be the basepoint of $BG$. For each $1$-simplex $e$ of $S$, define $f'$ on $e$ to be a loop in $BG$ whose homotopy class is $f(e) \in G$.
Since every $\sigma \in \SSS_2$ satisfies $f(\partial_1 \sigma) = f(\partial_2 \sigma) \cdot f(\partial_0 \sigma)$, we have that $f'$ extends to $\sigma$, and hence to the whole surface $S$.
Conversely, suppose that $S$ is a compact connected orientable surface with $\partial S = S^1$ and $f' \colon S \to BG$ is a continuous map such that $f'|_{\partial S}$ is a loop whose homotopy class is $x$.
Let $\SSS$ be a $\Delta$-complex structure such that $\partial S$ is a subcomplex of $(S,\SSS)$ having only one $1$-simplex.
 By Proposition 0.16 of \cite{H}, $f'$ is homotopic to a map $f''$ such that $f''$ sends every $0$-simplex of $S$ to the basepoint of $BG$. For each $e \in \SSS_1$, define $f(e)$ to be the homotopy class of the loop $f''|_e$ in $BG$.
Since $f''(\partial_2 \sigma) \cdot f''(\partial_0 \sigma)$ and $f''(\partial_1 \sigma)$ is homotopic for each $\sigma \in \SSS_2$, we have that $f(\partial_1 \sigma) = f(\partial_2 \sigma) \cdot f(\partial_0 \sigma)$, and hence $f$ is a $G$-labelling.
In summary, there exists $G$-surface of genus $g$ with boundary $x$ if and only if there exists a continuous map $f' \colon S \to x$ of genus $g$ with boundary $S^1$ such that $x$ is the homotopy class of $f' |_{\partial S}$.
Therefore,  Theorem \ref{thm 6.2} for the special case $\bG=\hG$ provides the geometric interpretation of the usual commutator length $\cl_G$.
\end{rem}

The proof of Theorem \ref{thm 6.2} is postponed to the next subsection. In the rest of this subsection, we deduce Theorem \ref{thm 6.1} from Theorem \ref{thm 6.2}.


The following lemma can be shown in a standard way; we omit the proof of it.

\begin{lem}
For $x \in [\hG,\bG]$,
\[\lim_{n \to \infty}  \frac{\| nx\|'_\ZZ}{n} = \| x\|'.\]
\end{lem}

We are now ready to prove Theorem \ref{thm 6.1}.

\begin{proof}[Proof of Theorem~$\ref{thm 6.1}$]
Let $n$ be a positive integer. Since $x \in [\hG,\bG]$, $n x$ is an integral $(\hG,\bG)$-boundary and hence there exists $c \in C_2'$ with integral coefficient such that $n x = \partial c$ and $\| nx \|'_\ZZ = \| c\|_1$. Define a chain $c' \in C'_2$ by
\[c' = c - \sum_{i=1}^{n-1} (x, x^i).\]
Then we have $\| c'\|_1 \le \| nx \|'_\ZZ + n -1$ and $\partial c' = x^n$. Thus it follows from Lemma \ref{lem 6.2} that there exists a $(\hG,\bG)$-simplicial surface $(S,\SSS, f)$ with boundary $x^n$ such that the number of $2$-simplices of $(S,\SSS)$ is at most $\| c'\|_1$.
Let $g$ be the genus of $S$, and let $p,e,s$ be the numbers of $0$-simplices, $1$-simplices, and $2$-simplices of $(S,\SSS)$.
Then the following hold:

\begin{enumerate}[(1)]
\item By observing the Euler characteristic of $S$, we have
\[s - e + p = 1 - 2g.\]

\item Every $1$-cell not contained in the boundary of $S$ appears twice as faces of $2$-cells. Since the boundary of $S$ has one $1$-cell, we have
\[1 + 2(e - 1) = 3s.\]
\end{enumerate}

Thus we have
\[ \| c'\|'_1 \ge s = 4g + 2p - 3 \ge 4g - 1 \ge 4 \cdot \cl_{\hG, \bG}(x^n) - 1.\]
Here we use Theorem \ref{thm 6.2} to deduce the last inequality. Since $\| nx\|_\ZZ' + n - 1 \ge \| c'\|'_1$, we have
\[ \| nx\|'_\ZZ \ge 4 \cdot \cl_{\hG, \bG}(x^n) - n .\]
By multiplying $n^{-1}$ and taking their limits, we have
\[\| x\|' \ge 4 \cdot \scl_{\hG,\bG}(x) - 1\]
for every $x \in [\hG,\bG]$. Thus we have
\[\| x^m\|' \ge 4 \cdot \scl_{\hG,\bG}(x^m) - 1 = 4 m \cdot \scl_{\hG,\bG}(x) - 1.\]
Multiplying $m^{-1}$ and taking the limits, we have
\[\fill_{\hG,\bG}(x) \ge 4 \cdot \scl_{\hG,\bG}(x).\]
This completes the proof of Theorem \ref{thm 6.1}.
\end{proof}

\subsection{Proof of Theorem \ref{thm 6.2}}

We start the proof of Theorem \ref{thm 6.2}. Theorem \ref{thm 6.2} is deduced from Propositions \ref{prop 6.5} and  \ref{prop 6.6}. Throughout this section, we write
$g_{\textrm{min}}$ to indicate the minimum of the genus of a $(\hG,\bG)$-simplicial surface with boundary $x$. The following proposition implies the inequality $g_{\textrm{min}} \le \cl_{\hG,\bG}(x)$.

\begin{prop} \label{prop 6.5}
Let $x \in [\hG,\bG]$ and set $m = \cl_{\hG,\bG}(x)$. Then there exists a connected $(\hG,\bG)$-simplicial surface with boundary $x$ whose genus is $m$.
\end{prop}
\begin{proof}
Let $x \in [\hG,\bG]$. Let $c_1, \cdots, c_m$ be $(\hG,\bG)$-commutators satisfying $x = c_1 \cdots c_m$. Our goal is to show that there exists a $(\hG,\bG)$-simplicial surface with boundary $x$ whose genus is $m$. Thus it suffices to find a $(\hG,\bG)$-triangulation of the $(4m+1)$-gon depicted in Figure \ref{fig 4}.
To construct this, we first embed three $2$-simplices to the part $\xrightarrow{h_i} \xrightarrow{g_i} \xleftarrow{h_i} \xleftarrow{g_i}$ as is depicted in the left of Figure \ref{fig 5}; after that, we embed $m-2$ triangles as is depicted in the right of Figure \ref{fig 5}.
Since $[\hG,\bG] \subset \bG$, we have that this $G$-labelling is a $(G,N)$-labelling. This completes the proof.
\end{proof}

\begin{figure}[t]
\begin{picture}(120,130)(0,0)
\put(80,0){\vector(-1,0){40}}

\put(40,0){\vector(-1,1){25}}
\put(15,25){\vector(-1,3){10}}
\put(25,85){\vector(-2,-3){20}}
\put(50,98){\vector(-2,-1){25}}

\put(70,98){\vector(2,-1){25}}
\put(95,85){\vector(2,-3){20}}
\put(105,25){\vector(1,3){10}}
\put(80,0){\vector(1,1){25}}

\put(60,-10){\footnotesize $x$}
\put(9,7){\footnotesize $h_m$}
\put(-7,36){\footnotesize $g_m$}
\put(0,73){\footnotesize $h_m$}
\put(28,96){\footnotesize $g_m$}

\put(83,96){\footnotesize $h_1$}
\put(107,73){\footnotesize $g_1$}
\put(113,35){\footnotesize $h_1$}
\put(97,7){\footnotesize $g_1$}

\put(55,98){\footnotesize $\cdots$}

\end{picture}
\caption{} \label{fig 4}
\end{figure}

On the other hand, the inequality $g_{\textrm{min}} \ge \cl_{\hG,\bG}(x)$ follows from the following proposition:

\begin{prop} \label{prop 6.6}
Let $x \in \hG$ and let $n$ be a positive integer. Assume that there exists a $(\hG,\bG)$-simplicial surface with boundary $x$ whose genus is $n$. Then there exist $n$ $(\hG,\bG)$-commutators $c_1$, $\cdots$, $c_n$ such that $x = c_1 \cdots c_n$. In particular, if there exists a $(\hG,\bG)$-simplicial surface with boundary $x$, then $x$ is contained in $[\hG,\bG]$.
\end{prop}

\begin{proof}
Throughout this proof, we assume that $BN$ is a subspace of $BG$ such that the map $\pi_1(BN) \hookrightarrow \pi_1(BG)$ induced by the inclusion $BN \hookrightarrow BG$ corresponds to the inclusion $N \hookrightarrow G$.
Let $(S,\SSS, f)$ be a $(\hG,\bG)$-simplicial surface with boundary $x$ such that the genus of $S$ is $n$. Let $(S', \SSS')$ be the subcomplex of $(S, \SSS)$ defined as follows: The set of $0$-simplices $\SSS'_0$ of $\SSS'$ is equal to $\SSS_0$. The set of $1$-simplices $\SSS'_1$ of $\SSS'$ is the 1-simplices of $\SSS$ labelled by an element of $N$.
The set of $2$-simplices $\SSS'_2$ is the simplices $\sigma$ satisfying that all of $\partial_0 \sigma$, $\partial_1 \sigma$, and $\partial_2 \sigma$ are contained in $\SSS'_1$.
In the same way as in Remark \ref{rem geometric interpretation of cl_G}, we can construct a continuous map $f' \colon S \to BG$ satisfying the following properties:
\begin{itemize}
\item Every $0$-simplices of $(S,\SSS)$ are mapped to the basepoint of $BG$.

\item For every $1$-simplex $e$ of $(S,\SSS)$, the homotopy class of $f'|_e$ is $f(e)$.

\item $f(S') \subset BN$
\end{itemize}

Let $\sigma$ be a 2-simplex of $(S,\SSS)$ not contained in $\SSS'_2$. Since $f$ is a $(G,N)$-labelling, $\partial_1 \sigma$ is not contained in $S'$, either $\partial_0 \sigma$ or $\partial_2 \sigma$ is not contained in $S'$, and the other is contained in $S'$. In particular, $\sigma$ has exactly two $1$-faces not contained in $S'$.
Consider the line segment in $S$ which connects the two midpoints of the two $1$-faces of $\sigma$ not contained in $S'$, and let $C$ be the union of these line segments.
Then $C$ is a 1-dimensional CW-complex contained in $S$, and satisfies the following properties:

\begin{enumerate}[(1)]
\item[(A)] Every connected component of $C$ is homeomorphic to $S^1$.

\item[(B)] For each connected component $\gamma$ of $C$, the inclusion $\gamma \hookrightarrow S$ is homotopic to a map factored through $S'$.

\item[(C)] $S'$ is a deformation retract of $S - C$.
\end{enumerate}

\begin{figure}[t]
\begin{picture}(340,130)(10,0)

  \put(140,30){\vector(-1,0){120}}
  \put(140,30){\vector(-1,1){60}}
  \put(20,30){\vector(1,1){60}}

  \put(125,75){\vector(-3,1){43}}
  \put(125,75){\line(-3,1){45}}
  \put(140,30){\vector(-1,3){15}}

\put(35,75){\vector(3,1){43}}
\put(35,75){\line(3,1){45}}
\put(20,30){\vector(1,3){15}}

\put(20,30){\line(-1,-3){5}}
\put(145,15){\vector(-1,3){5}}

\put(68,18){\tiny $[\gk_i, h_i]$}

\put(16,55){\tiny $h_i$}
\put(51,89){\tiny $g_i$}

\put(100,89){\tiny $h_i$}
\put(137,55){\tiny $g_i$}

\put(51,52){\tiny $h_i g_i$}
\put(93,52){\tiny $g_i h_i$}

\put(300,0){\vector(1,1){40}}
\put(300,0){\vector(-1,0){40}}
\put(300,0){\vector(1,2){40}}

\put(300,0){\vector(-1,3){40}}
\put(300,0){\vector(-1,1){80}}
\put(300,0){\vector(-2,1){80}}

\put(340,40){\vector(0,1){38}}
\put(340,40){\line(0,1){40}}
\put(340,80){\line(-3,4){20}}

\put(260,120){\vector(-1,-1){40}}
\put(220,80){\vector(0,-1){40}}
\put(220,40){\vector(1,-1){40}}

\put(323,13){\tiny $[g_1, h_1]$}
\put(210,13){\tiny $[g_m, h_m]$}
\put(170,58){\tiny $[g_{m-1},h_{m-1}]$}
\put(343,55){\tiny $[\gk_2,h_2]$}

\put(259,-9){\footnotesize $x$}

\put(265,115){$\cdots$}
\end{picture}
\caption{} \label{fig 5}
\end{figure}

Indeed, we will first show (A). Since every $1$-cell of $S$ is contained in at most two $2$-faces, there does not exist any vertex in $C$ with degree greater than 2. Thus $C$ is a compact 1-dimensional manifold which may have a non-empty boundary. Suppose that $C$ has a non-empty boundary and $v$ is a point in $\partial C$. If $v$ is not contained in the boundary of $S$, then $v$ is contained in exactly two $2$-cells $\sigma$ and $\tau$ of $S$. Since $\sigma$ and $\tau$ are not mapped to $BN$ by the map $f'$, $v$ is contained in the line segments of $\sigma$ and $\tau$.
Thus $v$ is not a boundary point in $C$.
Therefore $\partial C$ is contained in $\partial S$. Since $\partial S$ has only one $1$-simplex, we have that $\partial C$ consists of one point. Since $C$ is a compact $1$-dimensional manifold, this is impossible. Hence we have $\partial C = \emptyset$.
This completes the proof of (A).

Next we show (B). Let $\gamma$ be a connected component of $C$.
There are two ways to verify (B).
Since $\gamma$ is a simple closed curve in the orientable surface $S$, the normal bundle of $\gamma$ is trivial and hence we can slide $\gamma$ to $S'$. The second one is more concrete. Let $\sigma$ be a 2-cell not contained in $S'$. Since $f$ is a $(\hG,\bG)$-labelling, we can slide the line segment to $S'$ along the directions of each $1$-simplex containing an endpoint (see Figure \ref{fig 6}). Combining these deformations for all the line segments forming $C$, we have a homotopy from $\gamma \hookrightarrow S$ to a map factored through $S'$. Considering similar deformations, we have condition (C).

We may consider each connected component $\gamma$ of $C$ is a simple closed curve whose basepoint is mapped to a vertex in $S'$ by the homotopy considered in (B).

We define simple closed curves $\gamma_1, \cdots, \gamma_n$ as follows: recall that a simple closed curve in a topological surface $S$ is  \emph{non-separating} if the number of connected components of $S - \gamma$ is not larger than the number of connected components of $S$. First, if $C$ has a connected component $\gamma$ which is non-separating in $S$, then set $\gamma_1 = \gamma$ and $C_1 = C - \gamma$. Next if $C_1$ has a connected component which is non-separating in $S - \gamma_1$, then let $\gamma_2$ denote the connected component and set $C_2 = C_1 - \gamma_2$. Iterating this, we have simple closed curves $\gamma_1, \cdots, \gamma_k$ such that each $\gamma_i$ is a non-separating simple closed curve in $C - (\gamma_1 \cup \cdots \gamma_{i-1})$, and every connected component of $C - (\gamma_1 \cup \cdots \cup \gamma_k)$ is separating in $S - (\gamma_1 \cup \cdots \cup \gamma_k)$. Since the genus of $S$ is $n$, we have that $k \le n$, and hence $S - C$ has $n-k$ handles. If $k < n$, there exists a non-separating simple closed curve $\gamma_{k+1}$ in $S - C$. If $k + 1 \ne n$, we can further take a non-separating simple closed curve $\gamma_{k+2}$ in $S - (C \cup \gamma_{k+1})$. Iterating this, we can take simple closed curves $\gamma_{k+1}, \cdots, \gamma_n$
such that $S - (\gamma_1 \cup \cdots \cup \gamma_n)$ has no non-separating simple closed curves. For each $i$ greater than $k$ we take a basepoint of $\gamma_i$
in such a way that it is mapped to a vertex in $S'$ by the deformation retract considered in (C).

\begin{figure}[t]
\begin{picture}(270,100)(0,0)
\put(10,10){\vector(1,2){40}}
\put(90,10){\vector(-1,2){40}}
\put(90,10){\vector(1,2){40}}
\put(170,10){\vector(-1,2){40}}
\put(170,10){\vector(1,2){40}}
\put(250,10){\vector(-1,2){40}}

\put(-10,10){\line(1,0){280}}
\put(-10,90){\line(1,0){280}}

\multiput(30,50)(40,0){6}{\circle*{3}}

\multiput(26,58)(80,0){3}{\vector(1,2){6}}
\multiput(74,58)(80,0){3}{\vector(-1,2){6}}

\put(275,46){$C$}

\linethickness{1pt}

\put(-10,50){\line(1,0){280}}

\end{picture}
\caption{} \label{fig 6}
\end{figure}

By the classification theorem of compact surfaces, there exists a homeomorphism from $S$ to the surface depicted in Figure \ref{fig 7} which maps $\gamma_1, \cdots, \gamma_n$ to the simple closed curves as are depicted by red curves in Figure \ref{fig 7}.
Define simple closed curves $\delta_1, \cdots, \delta_n$ and paths $\alpha_1, \cdots, \alpha_n$ as are depicted in Figure \ref{fig 7}.
Here we assume that the basepoint of $\delta_i$ coincides with the basepoint of $\gamma_i$.
The basepoint of $S$ is the vertex contained in the boundary of $S$, and $\alpha_i$ connects the basepoint of $S$ with the basepoint of $\gamma_i$ and $\delta_i$.

Let $\overline{\alpha}_i$ be the reverse of the path $\alpha_i$.
Then $f'(\alpha_i \cdot \gamma_i \cdot \overline{\alpha}_i)$ is a loop in $B{\hG}$, and let $h_i$ be the element of $\hG$ associated to $f'(\alpha_i \cdot \gamma_i \cdot \overline{\alpha}_i)$.
Similarly, let $\gk_i$ be the element of $\hG$ associated to $f'(\alpha_i \cdot \delta_i \cdot \overline{\alpha}_i)$. Then we have
\[x = [h_n^{-1}, g_n^{-1}] \cdots [h_1^{-1}, g_1^{-1}].\]
Thus it suffices to prove that $h_i  = [f'(\alpha_i \cdot \gamma_i \cdot \overline{\alpha}_i)]$ is an element of $\bG$.

For each $i$, there exists a map $\varphi_i \colon S \to S$ which satisfies the following:
\begin{enumerate}[(a)]
\item There exists a based homotopy from the identity of $S$ to $\varphi_i$.

\item $\varphi_i(\gamma_i) \subset S'.$

\item $\varphi_i$ sends the basepoint of $\gamma_i$ to a vertex in $S'$.
\end{enumerate}
Indeed, to see this, for $i \le k$, it follows from Proposition 0.16 of \cite{H} that the homotopy of $\gamma_i \hookrightarrow S$ sending $\gamma_i$ to $S'$ discussed in (B) extends to the homotopy $H_i$ of $S$ from the identity of $S$. Let $\varphi_i$ be the end of this homotopy $H_i$. For $i > k$, we can similarly prove the existence of $\varphi_i$ using the deformation retract from $S - C$ to $S'$ considered in (C).

By (a), we have
\[ h_i = f'_{\ast} ([\alpha_i \cdot \gamma_i \cdot \overline{\alpha}_i]) = f'_{\ast}\big( [\varphi_i (\alpha_i \cdot \gamma_i \cdot \overline{\alpha}_i)] \big) = [f' \circ \varphi_i \circ \alpha_i] \cdot [f' \circ \varphi_i \circ \gamma_i] \cdot [f' \circ \varphi_i \circ \alpha_i]^{-1}.\]
It follows from (c) that $f' \circ \varphi_i \circ \alpha_i$ is a loop in $B{\hG}$, and hence $[f' \circ \varphi_i \circ \alpha_i]$ is an element of $\hG$.
It follows from (b) that $f' \circ \varphi_i \circ \gamma_i$ is a loop in $B\bG$ and hence $[f' \circ \varphi_i \circ \alpha_i]$ is contained in $\bG$.
Since $\bG$ is normal, we have that $h_i$ is an element of $\bG$. This completes the proof.
\end{proof}

\begin{rem} \label{rem 1}
In the definition of $(\hG,\bG)$-simplicial surfaces, we do not admit a simplex $\sigma$ such that $\partial_1 \sigma$ is contained in $\bG$ but neither $\partial_0 \sigma$ nor $\partial_2 \sigma$ is contained in $\bG$. However, a similar geometric interpretation of $(\hG,\bG)$-commutator lengths holds if we admit such a triangle. To state it explicitly, we prepare some terminology.
We call a $G$-labelling $f$ on a simplicial surface $(S,\SSS)$ a \textit{pseudo-$(G,N)$-simplicial surface} if for every $\sigma \in \SSS_2$ one of $\partial_0 \sigma$, $\partial_1 \sigma$, and $\partial_2 \sigma$ is mapped to $N$.
Then in an almost the same argument as in the proof of Theorem \ref{thm 6.2},
it can be shown that the $(\hG,\bG)$-commutator length of $x \in [\hG,\bG]$ coincides with the minimum genus of a connected pseudo-$(\hG,\bG)$-surface with boundary $x$.
\end{rem}

\begin{figure}[t]
  \centering
  \includegraphics[width=12cm]{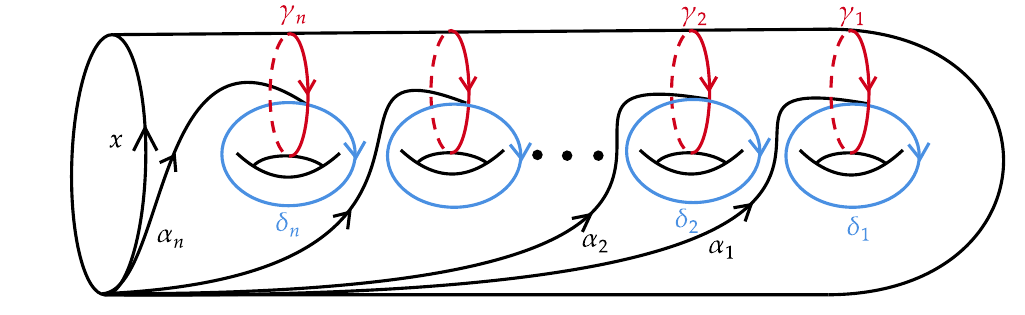}

\caption{} \label{fig 7}
\end{figure}

\section{Proof of the main theorem}\label{section Proof of the main theorem}

The goal of this section is to complete the proof of our generalization of Bavard's duality. See Theorem~\ref{thm 7.1} for the precise statement. This part is a straightforward generalization of the corresponding part of Bavard's original proof.

Let $\QQQ^{\mathrm{h}}(\bG)^{\hG}$ denote the space of homogeneous $\qad$-invariant quasimorphisms on $\bG$, and $\HH^1(\bG)^{\hG}$ the space of $\qad$-invariant homomorphisms from $\bG$ to $\RR$. Recall that $\QQQ_\bG = \QQQ_\bG(\hG)$ is the space of $\bG$-quasimorphisms on $\hG$ and $\HH^1_\bG = \HH^1_\bG(\hG)$ is the space of $\bG$-homomorphisms on $\hG$ (see Section 2).

\begin{thm} \label{thm 7.1}
For $a \in [\hG,\bG]$, the following equality holds:
\[\scl_{\hG,\bG}(a) = \frac{1}{2} \sup_{f \in \QQQ^{\mathrm{h}}(\bG)^{\hG} - \HH^1(\bG)^{\hG}} \frac{|f(a)|}{D(f)}.\]
\end{thm}
\begin{proof}
Let $f \in \QQQ^{\mathrm{h}}(\bG)^{\hG}$. Suppose $\cl_{\hG,\bG}(a) = m$ and let $c_1, \cdots, c_m$ be $(\hG,\bG)$-commutators such that $a = c_1 \cdots c_m$. Then, by Lemma \ref{lem 3.6}, we have
\[f(a) \sim_{(m-1)D(f)} f(c_1) + \cdots + f(c_m) \sim_{mD(f)} 0.\]
Thus we have $|f(a)| \le (2m-1) D(f) = (2 \cdot \cl_{\hG,\bG}(a) - 1) D(f)$. Therefore
\[|f(a)| = \frac{|f(a^n)|}{n} \le \frac{\cl_{\hG, \bG}(a^n)}{n} \cdot 2 D(f) - \frac{D(f)}{n}.\]
By taking the limit, we have an inequality
\[|f(a)| \le 2 D(f) \cdot \scl_{\hG,\bG}(a).\]
Thus we have
\[\scl_{\hG,\bG}(a) \ge \frac{1}{2}\sup_{f \in \QQQ^{\mathrm{h}}(\bG)^{\hG} - \HH^1(\bG)^{\hG}} \frac{|f(a)|}{D(f)} .\]

Next, we show the converse of the inequality. By Theorem \ref{thm 6.1} and Corollary \ref{cor 5.8}, we have
\[4 \cdot \scl_{\hG,\bG}(a) = \fill_{\hG,\bG}(a) = \lim_{n \to \infty} \frac{\| a^n\|'}{n} = \lim_{n \to \infty} \Big( \sup_{f \in \QQQ_\bG - \HH^1_\bG} \frac{|f(a^n)|}{n D''(f)}\Big). \]
For each pair $n$ and $m$ of positive integers, take $\fk_{n,m} \in \QQQ_\bG$ to satisfy
\[\sup_{f \in \QQQ_\bG - \HH^1_\bG} \frac{|f(a^n)|}{n D''(f)} \sim_{\frac{1}{m}} \frac{|\fk_{n,m}(a^n)|}{n D''(\fk_{n,m})}.\]
Hence, by Lemma 2.21 of \cite{Ca} and the definition of $D''$, we have
\[\| \fk_{n,m}|_\bG - \overline{\fk_{n,m}|_\bG}\| \le D(\fk_{n,m} |_\bG) \le D''(\fk_{n,m}).\]
Here $\overline{\fk_{n,m} |_\bG}$ is the homogenization of the quasimorphism $\fk_{n,m}|_\bG$ on $\bG$. Then
\[\frac{|\overline{\fk_{n,m} |_\bG} (a)|}{D''(\fk_{n,m})} = \frac{|\overline{\fk_{n,m}|_\bG}(a^n)|}{nD''(\fk_{n,m})} \sim_{\frac{1}{n}} \frac{\fk_{n,m}(a^n)}{n D''(\fk_{n,m})} \sim_{\frac{1}{m}} \sup_{f \in \QQQ_\bG - \HH^1_\bG} \frac{|f(a^n)|}{nD''(f)}.\]
Thus we have that
\[\lim_{n \to \infty} \frac{|\overline{\fk_{n,n} |_\bG}(a)|}{D''(\fk_{n,n})} = \lim_{n \to \infty} \Big( \sup_{f \in \QQQ_\bG - \HH^1_\bG} \frac{|f(a^n)|}{n D''(f)}\Big) = 4 \cdot \scl_{\hG,\bG}(a).\]
Therefore we have
\[\scl_{\hG,\bG}(a) \le \frac{1}{4} \sup_{f \in \QQQ_\bG - \HH^1_\bG} \frac{|\overline{f|_\bG}(a)|}{D''(f)} \le \frac{1}{2} \sup_{f \in \QQQ_\bG - \HH^1_\bG} \frac{|\overline{f|_\bG}(a)|}{D(\overline{f|_\bG})} \le \frac{1}{2} \sup_{f \in \QQQ^{\mathrm{h}}(\bG)^{\hG} - \HH^1(\bG)^{\hG}} \frac{|f(a)|}{D(f)}.\]
Here we use $D(\overline{f|_\bG}) \le 2 D(f |_\bG) \le 2 D''(f)$ and the fact that the restriction to $\bG$ of an $\bG$-quasimorphism is $\qad$-invariant (Lemma \ref{lem 3.3}).
\end{proof}

\section{On equivalences of $\scl_\hG$ and $\scl_{\hG,\bG}$}\label{section=equi_scl}

Two functions $\varphi$ and $\psi$ on a set $X$ are (bi-Lipschitzly) \textit{equivalent} on $X$ if there exists a positive number $C$ such that $C^{-1} \cdot \psi(x) \le \varphi(x) \le C \cdot \psi(x)$.
In this section, we 
provides a sufficient condition under which $\scl_\hG$ and $\scl_{\hG,\bG}$ are equivalent on $[\hG,\bG]$.
To state our result, we recall
the notion of virtually split exact sequences.

\begin{dfn}\label{defn=vs}
  A short exact sequence $1 \to \bG \to \hG \xrightarrow{\ppi} Q \to 1$ is said to \emph{virtually split}
  if there exist a subgroup $\Lambda$ of $Q$ of finite index and a homomorphism $s \colon \Lambda \to \hG$ such that $\ppi \circ s = \mathrm{id}_{\Lambda}$.
  We refer to such a pair $(s,\Lambda)$ as a \emph{virtual section} of $\ppi$.

\end{dfn}

Let $\hG$ be a group and $\bG$ a normal subgroup of $\hG$. Let $\ppi$ denote the projection $G \to \hG/\bG$.
If $\hG/\bG$ is a
virtually free group, including a finite group, or $\ppi \colon \hG \to Q:=\hG/\bG$ has a section homomorphism $s \colon
Q \to \hG$, then $1 \to \bG \to \hG \to 
Q \to 1$ virtually splits.



One of the applications of Theorem \ref{main thm} is to show
the following theorem on the equivalence of $\scl_\hG$ and $\scl_{\hG,\bG}$ on $[\hG,\bG]$.

\begin{thm} \label{thm scl equivalence}
Assume that the exact sequence $1\to \bG \to \hG \to
G/\bG
\to 1$ virtually splits.
Then for every $x \in [\hG,\bG]$, the following inequalities holds:
\[
\scl_\hG(x) \le \scl_{\hG,\bG}(x) \le 2 \cdot \scl_\hG(x).
\]
In particular, $\scl_\hG$ and $\scl_{\hG,\bG}$ are equivalent on $[\hG,\bG]$.
\end{thm}

Note that the first and second authors showed the above theorem in the case that $[\hG,\bG] = \bG$ and $G \to
\hG/\bG$ has a section homomorphism \cite{KK}.

\begin{rem} \label{remark non-equivalence example}
Let $\Sigma_l$ denote the orientable closed connected surface whose genus $l$ is at least $2$, and $\omega$ a symplectic form of $\Sigma_l$. Kawasaki and Kimura \cite{KK} showed that $\scl_{G,N}$ and $\scl_G$ are \emph{not} equivalent when $G = \Symp_0(\Sigma_l, \omega)$ and $N = \Ham(\Sigma_l, \omega)$. Here $\Symp_0(\Sigma_l, \omega)$ denotes the identity component of the group of symplectomorphisms on $(\Sigma_l, \omega)$, and $\Ham(\Sigma_l, \omega)$ denotes the group of Hamiltonian diffeomorphisms on $(\Sigma_l, \omega)$.
\end{rem}

\begin{rem} \label{rem 7.2}
On the other hand, the equivalence between $\scl_{\hG,\bG}$ and $\scl_\bG$ does \emph{not} hold in general even when the groups $\hG$ and $\bG$ are quite similar for instance, $[\hG:\bG]<\infty$. The first and second authors showed in \cite{KK} that $\scl_{B_3,P_3}$ and $\scl_{P_3}$ are not equivalent on $[P_3,P_3]$.
As an application of Theorem \ref{main thm}, we can show that in the following simple example, $\scl_{\hG,\bG}$ and $\scl_\bG$ are not equivalent on $[\bG,\bG]$.

Let $\bG$ be the free group $F_2$ generated by $\langle a, b \rangle$, and $\hG$ the semi-direct product $F_2 \rtimes (\ZZ/2\ZZ)$. Here, $\ZZ/2\ZZ$ means that the cyclic group $\{1,\tau\}$ of order $2$ and the action of $\ZZ/2\ZZ$ on $F_2$ is given by the automorphism exchanging $a$ and $b$.

Set $x = [a,b] \in [\bG,\bG]$. Then it is well-known that $\scl_\bG(x) = 1 / 2$. On the other hand, since
\[
(1,\tau) \cdot ([a,b], 1) \cdot (1, \tau) = ([b,a], 1) = ([a,b], 1)^{-1},
\]
every $\hG$-invariant homogeneous quasimorphism vanishes at $[a,b]$. Thus Theorem \ref{main thm} implies that $\scl_{\hG,\bG}([a,b]) = 0$. Hence, $\scl_{\hG,\bG}$ and $\scl_\bG$ are not equivalent on $[\hG,\bG]$. Note that in this example $\hG/\bG$ is not only finite but also there exists a section homomorphism $\hG/\bG \to \hG$.

See also Proposition~\ref{prop=freescl} for other examples of pairs $(\hG,\bG)$ such that $[\hG:\bG]<\infty$ and that $\scl_{\hG,\bG}$ and $\scl_\bG$ are not equivalent on $[\hG,\bG]$.
\end{rem}

Now we start the proof of Theorem \ref{thm scl equivalence}. We first prove the following generalization of some extension results of quasimorphisms by Ishida \cite{I} and Shtern \cite{Sh} (see also \cite[Theorem~1.10 and Theorem~6.13]{KKMM} for further generalizations of this proposition); if $f$ is furthermore homogeneous, then the following proposition is stated in Proposition~\ref{prop=extension}:

\begin{prop}[Extension theorem for $\hG$-invariant quasimorphisms, general version]\label{prop 7.3}
Assume that the exact sequence $1\to \bG \to \hG
\to Q \to 1$ virtually splits. Then, for every $\hG$-quasi-invariant quasimorphism $f$ on $\bG$, there exists a quasimorphism $f'$ on $\hG$ such that $f'|_\bG = f$ and $D(f') \le D(f) + 3 D'(f)$.
\end{prop}
To derive Theorem~\ref{thm scl equivalence} from Theorem~\ref{main thm}, we need to estimate the defect of the extended quasimorphism in our extension theorem; the inequality in Proposition \ref{prop 7.3} does the job.
To 
obtain the desired upper bound of the defect, we generalize Ishida's idea \cite[Lemma~3.1]{I}.
\begin{proof}[Proof of Proposition~$\ref{prop 7.3}$]
Let $\ppi\colon G\to Q$ be the projection as in the short exact sequence. Let $(s,\Lambda)$ be a virtual section of
$\ppi$.
Let $B$ be a finite subset of $Q$ such that the map $\Lambda \times B \to Q$, $(\lambda, b) \mapsto \lambda b$ is bijective. Let $s' \colon B \to \hG$ be a map such that $\ppi \circ s'(b) = b$ for every $b \in B$.
Define a map $t \colon Q \to \hG$ by $t(\lambda b) = s(\lambda) s'(b)$ for every $\lambda \in \Lambda$ and every $b \in B$.
Then $t$ is a set-theoretic section of $\ppi \colon \hG \to Q$, and satisfies $t(\lambda x) = s(\lambda) t(x)$ for every $\lambda \in \Lambda$ and every $x \in Q$.
For a $\qad$-invariant quasimorphism $f \colon \bG \to \RR$ on $\bG$, define a function $f' \colon \hG \to \RR$ on $\hG$ by
\[f'(\gk ) = \frac{1}{\# B} \sum_{b \in B} f \big( \gk \cdot t(b \cdot \ppi(\gk))^{-1} t(b)\big).\]
Here observe that $\gk \cdot t(b \cdot \ppi(\gk))^{-1} t(b) \in  \bG$. Indeed, we have
\[
\ppi\big(\gk \cdot t(b \cdot \ppi(\gk))^{-1} t(b)\big)=\ppi(\gk)(b \cdot \ppi(\gk))^{-1}b=e_Q.
\]
Our goal is to show that $f'$ is a quasimorphism which is an extension of $f$. If $\gk \in \bG$, then $\ppi(\gk ) = e_Q$. Therefore we have
\[f'(\gk ) = \frac{1}{\# B} \sum_{b \in B} f (g \cdot t(b)^{-1} t(b)) = f(\gk ).\]
Next we have
\begin{eqnarray*}
f'(\gk_1 \gk_2) & = & \frac{1}{\# B} \sum_{x \in B} f(g_1 g_2 \cdot t(x \cdot \ppi(g_1 g_2))^{-1} t(x))\\
& \sim_{D'(f)} & \frac{1}{\# B} \sum_{x \in B} f(t(x) \cdot g_1 g_2 \cdot t(x \cdot \ppi(g_1 g_2))^{-1}) \\
& = & \frac{1}{\# B} \sum_{x \in B} f(t(x) \cdot g_1 \cdot t(x \cdot \ppi(g_1))^{-1} \cdot t(x \cdot \ppi(g_1)) \cdot g_2 \cdot t(x \cdot \ppi(g_1 g_2))^{-1})\\
& \sim_{D(f)}& \frac{1}{\# B} \sum_{x \in B} \Big(f(t(x) \cdot g_1 \cdot t(x \cdot \ppi(g_1))^{-1}) + f(t(x \cdot \ppi(g_1)) \cdot g_2 \cdot t(x \cdot \ppi(g_1 g_2))^{-1})\Big)\\
& \sim_{2D'(f)}& \frac{1}{\# B} \sum_{x \in B} \Big( f(g_1 \cdot t (x \cdot \ppi(g_1))^{-1} t(x)) + f(g_2 \cdot t(x \cdot \ppi(g_1 g_2))^{-1} t(x \cdot \ppi(g_1)))\Big)\\
& = & f'(\gk_1) + \frac{1}{\# B}\sum_{x \in B} f\big(g_2 \cdot t \big((x \cdot \ppi(g_1)) \cdot \ppi(g_2) \big)^{-1} \cdot t(x \cdot \ppi(g_1))\big).
\end{eqnarray*}
Hence it suffices to prove
\[
\frac{1}{\# B}\sum_{x \in B}f \big(g_2 \cdot t\big((x \cdot \ppi(g_1)) \cdot \ppi(g_2)\big)^{-1} \cdot t(x \cdot \ppi(g_1)) \big) = f'(\gk_2).
\]
To see this, define a function $F \colon Q \to \RR$ by
\[
F(x) = f(g_2 \cdot t(x \cdot \ppi(g_2))^{-1} t(x)).
\]
It suffices to see
\[
\sum_{x \in B} F(x) = \sum_{x \in B} F(x \cdot \ppi(\gk_1)).
\]
Let $\lambda_1, \cdots, \lambda_r \in \Lambda$ such that $B \ppi(g_1) \subset \lambda_1 B \cup \cdots \cup \lambda_r B$. Define $C_i \subset B \ppi(\gk_1)$ by
\[
C_i = \{ x \in B \ppi(\gk_1) \; | \; \lambda_i^{-1} x \in B\}.
\]
Since the map $\Lambda \times (B \ppi(\gk_1)) \to Q$, $(\lambda, b) \mapsto \lambda b$ is also bijective, we have $B = \lambda_1^{-1} C_1 \sqcup \cdots \sqcup \lambda_1^{-1} C_r$. Since $t(\lambda x) = s(\lambda) t(x)$ and $s$ is a homomorphism, $F$ is left $\Lambda$-invariant. Thus, we have
\begin{eqnarray*}
\sum_{x \in B} F(x \cdot \ppi(\gk_1)) & = & \sum_{x \in B \ppi(\gk_1)} F(x) \\
& = & \sum_{i=1}^r \Big( \sum_{x \in C_i} F(x) \Big) \\
& = & \sum_{i=1}^r \Big( \sum_{x \in C_i} F(\lambda_i^{-1} x) \Big) \\
& = & \sum_{i=1}^r \Big( \sum_{x \in \lambda_i^{-1} C_i} F(x)\Big) \\
& = & \sum_{x \in B} F(x).
\end{eqnarray*}
This completes the proof.
\end{proof}

\begin{proof}[Proof of Theorem~$\ref{thm scl equivalence}$]
The inequality $\scl_\hG(x) \le \scl_{\hG,\bG}(x)$ is obvious. Hence, it suffices to show $\scl_{\hG,\bG}(x) \le 2 \cdot \scl_\hG(x)$. It follows from Theorem \ref{main thm} that there exists a $\hG$-invariant homogeneous quasimorphism $f$ such that
\[
\scl_{\hG,\bG}(x) - \epsilon \le \frac{1}{2} \cdot \frac{f(x)}{D(f)}.
\]
It follows from Proposition \ref{prop 7.3} that there exists a quasimorphism $f'$ on $\hG$ such that $f'|_\bG = f$ and $D(f') = D(f)$. Let $\overline{f'}$ be the homogenization of $f$.
Then, it is known that $D(\overline{f'}) \le 2 D(f)$ (see \cite[Lemma 2.58]{Ca}).
Since $f$ is a homogeneous quasimorphism, $\overline{f'}|_\bG = f$.
Hence, Theorem \ref{main thm} implies that
\[
\scl_{\hG,\bG}(x) - \epsilon \le \frac{1}{2} \cdot \frac{f(x)}{D(f)} \le \frac{\overline{f'}(x)}{D(\overline{f'})} \le 2 \cdot \scl_\hG(x).
\]
Since $\epsilon$ is an arbitrary positive number, this completes the proof.
\end{proof}

\section{On equivalences of $\cl_\hG$ and $\cl_{\hG,\bG}$}\label{section=equi_cl}

Throughout this section, let $q \colon G \to G/N$ denote the projection.
The purpose in this section is to show some equivalence results concerning $\cl_{\hG,\bG}$ and $\cl_\hG$.
Namely, we show the following two theorems:

\begin{thm} \label{thm cl equivalence 1}
Assume that the projection $G \to \hG/\bG$ has a section homomorphism $s \colon \hG/\bG \to \hG$. Then for each $x \in [\hG,\bG]$, the following inequalities hold:
\[ \cl_\hG(x) \le  \cl_{\hG,\bG}(x) \le 3 \cdot \cl_\hG(x). \]
\end{thm}

\begin{thm} \label{thm cl equivalence 2}
Assume that $Q = \hG/\bG$ is a finite group. Then there exists a positive number $C$ such that for every $x \in [\hG,\bG]$, the following inequalities hold:
\[ \cl_\hG(x) \le \cl_{\hG,\bG}(x) \le C \cdot \cl_\hG(x). \]
\end{thm}

\begin{cor} \label{cor 8.3}
If $\hG/\bG$ is finite or the projection $\ppi \colon \hG \to \hG/\bG$ has a section homomorphism, then $\cl_{\hG,\bG}$ and $\cl_\hG$ are equivalent on $[\hG,\bG]$.
\end{cor}

Later, we will generalize Theorem~\ref{thm cl equivalence 2} (see Theorem \ref{thm cl equivalence 3}).

Here we give some remarks on Theorem \ref{thm cl equivalence 2}. Recall that in the case of stable commutator lengths, the constant $C$ can be taken to be $2$ (Theorem \ref{thm scl equivalence}). On the other hand, our estimation of $C$ can be much larger. In fact, we do not know that there exists a constant $C$ independent from $\hG$ and $\bG$ such that $\cl_\hG \le C \cdot \cl_{\hG,\bG}$.

\begin{lem} \label{lem 8.4}
Let $\fk_1, \cdots, \fk_k, \gk_1, \cdots, \gk_k, \alpha_1, \cdots, \alpha_k, \beta_1, \cdots, \beta_k \in \hG$ with
\[
\ppi(\fk_i) = \ppi(\alpha_i)
\quad \textrm{and}\quad
\ppi(\gk_i) = \ppi(\beta_i)
\]
for each $ i \in \{1,\cdots,k\}$.
Then, $[\fk_1, \gk_1] \cdots [\fk_k, \gk_k] ([\alpha_1, \beta_1] \cdots [\alpha_k, \beta_k])^{-1}$ is contained in $[\hG,\bG]$. Moreover, the following inequality holds:
\[
\cl_{\hG,\bG}\Big( [\fk_1,\gk_1] \cdots [\fk_k,\gk_k] \big( [\alpha_1, \beta_1] \cdots [\alpha_k, \beta_k] \big)^{-1} \Big) \le 3k.
\]
\end{lem}
\begin{proof}
We prove this lemma by induction on $k$. The case $k = 0$ is obvious. Suppose $k = 1$, and set $f = \fk_1$, $g = \gk_1$, $\alpha = \alpha_1$, and $\beta = \beta_1$. Then, there exist $h_1, h_2 \in \bG$ such that $f = \alpha h_1$ and $g = \beta h_2$. Then we have
\begin{eqnarray*}
[\fk,\gk][\alpha, \beta]^{-1} & = & [\alpha h_1, \beta h_2] [\alpha, \beta]^{-1} \\
& = & \alpha h_1 \beta h_2 h_1^{-1} \alpha^{-1} h_2^{-1} \beta^{-1} \beta \alpha \beta^{-1} \alpha^{-1} \\
& = & \alpha h_1 \beta h_2 h_1^{-1}  \alpha^{-1} h_2^{-1} \alpha \beta^{-1} \alpha^{-1} \\
& = & \alpha \beta (\beta^{-1} h_1 \beta h_1^{-1} h_1 h_2 h_1^{-1} h_2^{-1} h_2 \alpha^{-1} h_2^{-1} \alpha) \beta^{-1} \alpha^{-1} \\
& = & (\alpha \beta) [\beta^{-1}, h_1] [h_1, h_2] [h_2, \alpha^{-1}] (\alpha \beta)^{-1}.
\end{eqnarray*}
Since $\cl_{\hG,\bG}$ is $\hG$-invariant, we have $\cl_{\hG,\bG}([\fk,\gk] [\alpha, \beta]^{-1}) \le 3$. Suppose $k \ge 2$, and set $\gamma = [\alpha_1, \beta_1] \cdots [\alpha_{k-1}, \beta_{k-1}]$. Then we have
\[[\fk_1, \gk_1] \cdots [\fk_k, \gk_k] ([\alpha_1, \beta_1] \cdots [\alpha_k, \beta_k])^{-1} = ([\fk_1, \gk_1] \cdots [\fk_{k-1}, \gk_{k-1}] \gamma^{-1}) \cdot (\gamma [\fk_k, \gk_k] [\alpha_k, \beta_k]^{-1} \gamma^{-1}).\]
By the $\hG$-invariance of $\cl_{\hG,\bG}$ and the inductive hypothesis, we conclude
\[\cl_{\hG,\bG} \Big( [\fk_1, \gk_1] \cdots [\fk_k, \gk_k] ([\alpha_1, \beta_1] \cdots [\alpha_k, \beta_k])^{-1}\Big) \le 3(k-1)+3=3k. \qedhere\]
\end{proof}

\begin{proof}[Proof of Theorem~$\ref{thm cl equivalence 1}$]
Let $s \colon \hG/\bG \to \hG$ be a section homomorphism of $\ppi \colon \hG \to \hG/\bG$. Then, it follows from Lemma \ref{lem 8.4} that $\cl_\hG(x \cdot s \circ \ppi(x)^{-1}) \le 3 \cdot \cl_{\hG,\bG}(x)$. In particular, if $x \in [\hG,\bG] \subset \bG$, then $\ppi(x) = 1_\hG$ and hence
\[ \cl_\hG(x) = \cl_\hG(x \cdot s\circ \ppi(x)^{-1}) \le 3 \cdot \cl_{\hG,\bG}(x). \qedhere \]
\end{proof}

Next we prove Theorem \ref{thm cl equivalence 2}. For the proof, we need the following proposition:

\begin{prop} \label{prop 8.5}
If $\hG$ is finitely generated and $\hG / \bG$ is finite, then the group $[\hG,\hG] / [\hG,\bG]$ is finite.
\end{prop}

In what follows, we will prove Proposition~\ref{prop 8.5}. The proof of the following proposition is straightforward, and hence omitted:

\begin{prop}\label{77}
Let $\Gamma$ be a group and $f \colon \bG \to \Gamma$ a group homomorphism. Then $f$ is $\ad$-invariant if and only if $[\hG,\bG] \subset \mathrm{Ker}(f)$.
\end{prop}

For a pair $A$ and $B$ of abelian groups, we write $\Hom(A,B)$ to mean the abelian group of homomorphisms from $A$ to $B$.
Proposition \ref{77} implies the following corollary.

\begin{cor}
$\HH^1(\bG)^{\hG} = \Hom(\bG / [\hG,\bG], \RR)$.
\end{cor}

\begin{cor}\label{warizan isom}
Assume that the exact sequence $1\to \bG \to \hG \to \hG/\bG \to 1$ virtually splits.
Then the map $p^\ast\colon\Hom(\bG / ([\hG,\hG] \cap \bG), \RR) \to \Hom(\bG / [\hG,\bG] , \RR)$ induced by the natural projection $p \colon \bG /[\hG,\bG] \to \bG / ([\hG,\hG] \cap \bG)$ is an isomorphism.
\end{cor}
\begin{proof}
By Proposition \ref{prop 7.3}, every $\ad$-invariant homomorphism $\bG \to \RR$ can be extended to $\hG$ as a homomorphism.
Hence the map $\Hom(G/ [\hG,\hG], \RR) \to \Hom(\bG / [\hG,\bG] , \RR)$ is surjective.
Since this map is a composition of
\[
\Hom(G/ [\hG,\hG], \RR) \to \Hom(\bG / ([\hG,\hG] \cap \bG), \RR) \xrightarrow{p^{\ast}} \Hom(\bG / [\hG,\bG], \RR),
\]
we have that $p^{\ast}$ is surjective. Also, since $p$ is surjective, $p^{\ast}$ is injective.
\end{proof}

\begin{cor}\label{warizan finite}
If $\bG$ is a finitely generated group and the exact sequence $1\to \bG \to \hG \to \hG/\bG \to 1$ virtually splits,
then the group $([\hG,\hG] \cap \bG) / [\hG,\bG]$ is finite.
\end{cor}
\begin{proof}
Set $K = ([\hG,\hG] \cap \bG) / [\hG,\bG]$. Then, $K$ is the kernel of the projection $p \colon \bG / [\hG,\bG] \to \bG /([\hG,\hG] \cap \bG)$. Since $\RR$ is an injective abelian group, we have the following short exact sequence:
\[
0 \to \Hom(\bG / ([\hG,\hG] \cap \bG), \RR) \xrightarrow{p^{\ast}} \Hom(\bG / [\hG,\bG], \RR) \to \Hom(K, \RR) \to 0.
\]
Since $p^{\ast}$ is an isomorphism (Corollary \ref{warizan isom}), we have $\Hom(K, \RR) = 0$.

Since $\bG$ is finitely generated, $\bG / [\hG,\bG]$ is a finitely generated abelian group. Since $K \subset \bG / [\hG,\bG]$, we have that $K$ is a finitely generated abelian group. If $K$ is not finite, then $K$ has $\ZZ$ as a direct summand, and hence $\Hom(K, \RR) \ne 0$.
This is a contradiction.
\end{proof}

\begin{proof}[Proof of Proposition~$\ref{prop 8.5}$]
Since $\hG$ is finitely generated and $\hG/\bG$ is finite, we have that $\bG$ is finitely generated.
Hence, Corollary \ref{warizan finite} implies that $([\hG,\hG] \cap \bG)/ [\hG,\bG]$ is finite.
Since $[\hG,\hG] / ([\hG,\hG] \cap \bG)$ embeds into $\hG/\bG$, it is a finite group.
Therefore, we obtain the conclusion.
\end{proof}

We are now ready to
define the constant  $C$ in Theorem \ref{thm cl equivalence 2}.
In the rest of this paper, we write $Q$ 
for $\hG/\bG$, and assume that $Q$ is a finite group. Let $s \colon Q \to \hG$ be a set-theoretic section of $\ppi \colon \hG \to Q$, and $\hG(s)$ the subgroup of $\hG$ generated by $s(Q)$.

\begin{lem}
The group $W(s) = [\hG(s), \hG(s)] / ([\hG(s), \hG(s)] \cap [\hG,\bG])$ is finite.
\end{lem}
\begin{proof}
Set $\bG(s) = \hG(s) \cap \bG$. Since $\hG(s) / \bG(s)$ can be embedded into $\hG/\bG$, we have that $\hG(s) / \bG(s)$ is a finite group. Since $\hG(s)$ is finitely generated, it follows from Proposition \ref{prop 8.5} that $[\hG(s), \hG(s)] / [\hG(s), \bG(s)]$ is finite. Since $[\hG(s), \bG(s)] \subset [\hG(s), \hG(s)] \cap [\hG,\bG]$, we have that $[\hG(s), \hG(s)] / [\hG(s), \hG(s)] \cap [\hG,\bG]$ is finite.
\end{proof}

Set $M(s) = \# W(s)$.
Define a non-negative number $C(s)$ to be the maximum of the following finite subset of $\RR$:
\[
\left\{ \frac{\cl_{\hG,\bG}([\alpha_1, \beta_1] \cdots [\alpha_k, \beta_k])}{k} \; \middle| \; 1 \le k \le M(s), \alpha_1, \cdots , \beta_k \in s(Q), [\alpha_1, \beta_1] \cdots [\alpha_k, \beta_k] \in [\hG,\bG] \right\}
\]
Then, Theorem \ref{thm cl equivalence 2} immediately follows from the following theorem.

\begin{thm}[Precise form of Theorem \ref{thm cl equivalence 2}] \label{thm 8.11}
Assume that $Q = \hG/\bG$ is finite, and let $s \colon Q \to \hG$ be a set-theoretic section of $\ppi \colon \hG \to Q$. Then, for every $x \in [\hG,\bG]$, the inequality $\cl_{\hG,\bG}(x) \le (C(s) + 3) \cdot \cl_\hG(x)$ holds.
\end{thm}

We will employ the following lemma for the proof of Theorem~\ref{thm 8.11}:
\begin{lem} \label{lem 8.12}
Let $W$ be a finite group, set $M = \# W$, and let $w_1, \cdots, w_M \in W$. Then there exists a pair $(i,j)$ of elements in $\{ 1, \cdots, M\}$ such that $i \le j$ and $w_i w_{i+1} \cdots w_j = e_W$.
\end{lem}
\begin{proof}
For each $k \in \{ 0,1, \cdots, M\}$, set $v_k = w_1 \cdots w_k$. In particular, we have $v_0 = e_W$.
Then, the pigeonhole principle implies that there exists a pair $(i', j)$ of elements in $\{ 0,1, \cdots, M\}$ such that $i' < j$ and $v_{i'} = v_j$.
By setting $i = i' + 1$, we have completed the proof.
\end{proof}

We are now ready to prove Theorem \ref{thm 8.11}.

\begin{proof}[Proof of Theorem~$\ref{thm 8.11}$]
Let $x \in [\hG,\bG]$ and set $n = \cl_\hG(x)$. Then, there exist $\fk_1, \cdots, \fk_n, \gk_1, \cdots, \gk_n \in \hG$ such that $x = [\fk_1, \gk_1] \cdots [\fk_n, \gk_n]$. Set $\alpha_i = s \circ \ppi (\fk_i)$, $\beta_i = s \circ \ppi (\gk_i)$, $w_i = [\alpha_i, \beta_i]$, and $y = w_1 \cdots w_n = [\alpha_1, \beta_1] \cdots [\alpha_n, \beta_n]$. It follows from Lemma \ref{lem 8.4} that $xy^{-1} \in [\hG,\bG]$ and $\cl_{\hG,\bG} (xy^{-1}) \le 3n$.
Note that $x, xy^{-1} \in [\hG,\bG]$ implies that $y \in [\hG,\bG]$.
Hence it suffices to show
\[
\cl_{\hG,\bG}(y) \le C(s) \cdot n.
\]
To show this, we use the following claim:

\vspace{2mm} \noindent
{\bf Claim.} For a pair $r$ and $a$ of elements in $\hG$, let $r^a$ denote $a r a^{-1}$. Then, $y$ can be written as
\[
y = r_1^{a_1} \cdots r_l^{a_l}
\]
which satisfies the following conditions:
\begin{enumerate}[(1)]
\item Each $r_j$ is a product of $m_j$ $(s(Q), s(Q))$-commutators. Here, $m_j \le M(s)$.

\item $m_1 + \cdots + m_l = n$

\item $a_j \in \hG(s)$. 
\end{enumerate}

We prove the claim by induction on $n$. If $n \le M(s)$, then it suffices to set $r_1 = y$ and $a_1 = e_\hG$. Suppose $n > M(s)$. Since
\[W(s) = [\hG(s), \hG(s)] / ([\hG,\bG] \cap [\hG(s), \hG(s)])\]
is a finite group, Lemma \ref{lem 8.12} implies that there exists a pair $(i,j)$ of elements in $\{ 1, \cdots, n\}$ such that $i \le j$ and $r_1 = w_i \cdots w_j \in [\hG,\bG] \cap [\hG(s), \hG(s)]$. Then we have
\[
y = w_1 \cdots w_n = r_1^{w_1 \cdots w_{i-1}} \cdot (w_1 \cdots w_{i-1} w_{j+1} \cdots w_n).
\]
By applying the inductive hypothesis to $w_1, \cdots, w_{i-1} w_{j+1} \cdots w_n$, we have completed the proof of Claim.
Then, we have that $\cl_{\hG,\bG}(r_j) \le C(s) \cdot m_j$.
Indeed, if we set
\[r_j=[\alpha_1,\beta_1][\alpha_2,\beta_2]\cdots[\alpha_{m_j},\beta_{m_j}]\,\, (\alpha_1,\alpha_2,\ldots,\alpha_{m_j},\beta_1,\beta_2,\ldots,\beta_{m_j}\in s(Q)),\]
then we have
\begin{align*}
\cl_{\hG,\bG}(r_j) &\leq  \cl_{\hG,\bG}\left([\alpha_1,\beta_1][\alpha_2,\beta_2]\cdots[\alpha_{m_j},\beta_{m_j}]\right)\\
& \leq\frac{\cl_{\hG,\bG}\left([\alpha_1,\beta_1][\alpha_2,\beta_2]\cdots[\alpha_{m_j},\beta_{m_j}]\right)}{m_j}\cdot m_j \\
& \leq C(s)\cdot m_j.
\end{align*}
Hence by the conjugacy invariance of $\cl_{\hG,\bG}$,
\[\cl_{\hG,\bG}(y) \le \sum_{j=1}^l C(s) \cdot m_j = C(s) \cdot n.\]
This completes the proof of Theorem \ref{thm 8.11}.
\end{proof}
In what follows, we discuss two applications of our theorems.
As the first application, let us consider the equivalence problem of $\cl_{\hG}$ and $\cl_{\hG,\bG}$ on braid groups.
Let $B_n$ and $P_n$ denote the braid group and the pure braid group on $n$ strands, respectively.
Note that there exists an exact sequence
\[ 1 \to P_n \to B_n \to \mathfrak{S}_n \to 1, \]
where $\mathfrak{S}_n$ denotes the symmetric group of degree $n$.
It is known that the index homomorphism $B_n \to \ZZ$ defined by $\sigma_i \mapsto 1$, where $\sigma_i$ is the $i$th Artin generator, is the abelianization map of $B_n$.
Hence, we have an exact sequence
\[ 1 \to [B_n,B_n] \to B_n \to \ZZ \to 1. \]

\begin{eg}\label{pure braid}
Since $P_n$ is a normal subgroup of finite index of $B_n$, Theorem \ref{thm cl equivalence 2} implies that $\cl_{B_n}$ and $\cl_{B_n,P_n}$ are equivalent on $[B_n,P_n]$.
\end{eg}

\begin{eg}\label{commutator index sum}
If $\hG/\bG$ is a free group, then the quotient map $G \to \hG / \bG $ has a section homomorphism. Hence, Theorem \ref{thm cl equivalence 1} implies that $\cl_{\hG,\bG}$ and $\cl_\hG$ are equivalent on $[\hG,\bG]$.
For example, we have that $\cl_{B_n}$ and $\cl_{B_n,[B_n, B_n]}$ are equivalent on $[B_n,[B_n,B_n]]$
since $B_n / [B_n, B_n] \cong \ZZ$ is a free group.
Note that Kawasaki and Kimura \cite{KK} proved
$\scl_{B_n}$ and $\scl_{B_n, [B_n, B_n]}$ are equivalent on $[B_n,[B_n,B_n]]$, provided that $n \ge 5$.
\end{eg}

We also have the followng exact sequence:
\[ 1 \to P_n \cap [B_n, B_n] \to B_n \to \mathfrak{S}_n \times \ZZ \to 1. \]
Although this example does not satisfy the assumption of Theorem \ref{thm cl equivalence 1} or Theorem \ref{thm cl equivalence 2}, we in fact have the following result.
\begin{prop}\label{pure braid and commutator}
For every positive integer $n$, $\cl_{B_n}$ and $\cl_{B_n,P_n\cap[B_n,B_n]}$ are equivalent on $[B_n,P_n\cap[B_n,B_n]]$.
\end{prop}

To prove Proposition \ref{pure braid and commutator}, we provide the following theorem:


\begin{thm}\label{thm cl equivalence 3}
Let $\bG$ be a normal subgroup of a group $\hG$ and $\ppi \colon \hG \to \hG/\bG$ be the quotient map.
Assume that there exists a virtual section $(s,\Lambda)$ of $\ppi$ such that the image of $s$ is contained in the center $Z(\hG)$ for $\hG$.
Then, $\cl_{\hG,\bG}$ and $\cl_\hG$ are equivalent on $[\hG,\bG]$.
\end{thm}

Note that Theorem \ref{thm cl equivalence 3} is a generalization of Theorem \ref{thm=cl} (2) since the trivial subgroup is in the center.
In order to prove Theorem \ref{thm cl equivalence 3}, we employ the following lemma.

\begin{lem}\label{lem cl equivalence 3}
Let $\Lambda$ be a normal subgroup of $\hG/\bG$.
If the restriction of the quotient map $\ppi\colon \hG \to \hG / \bG $ to $\ppi^{-1}(\Lambda)$ has a section homomorphism whose image is contained in the center for $\hG$,
then $[\hG,\bG]=[\hG,\ppi^{-1}(\Lambda)]$ and $\cl_{\hG,\ppi^{-1}(\Lambda)}(x)=\cl_{\hG,\bG}(x)$ for any $x\in[\hG,\bG]$.
\end{lem}

\begin{proof}
Let $s$ be a section homomorphism of $\ppi|_{\ppi^{-1}(\Lambda)}$ and set $t = s \circ \ppi \colon \ppi^{-1}(\Lambda) \to s(\Lambda)$. Take $\alpha \in \hG$ and $\beta \in \ppi^{-1}(\Lambda)$, and set $b = t(\beta)^{-1}\beta \in \bG$.
Then, since $t(\beta) \in s(\Lambda) \subset Z(\hG)$, we have
\[[\alpha,\beta]=\alpha t(\beta) b \alpha^{-1} b^{-1} t(\beta)^{-1} = \alpha b \alpha^{-1}b^{-1} =[\alpha,b]. \]
Thus, $[\hG,\bG]=[\hG,\ppi^{-1}(\Lambda)]$ and $\cl_{\hG,\bG}=\cl_{\hG,\ppi^{-1}(\Lambda)}$.
\end{proof}

\begin{proof}[Proof of Theorem~$\ref{thm cl equivalence 3}$]
Take the normal core $\Lambda_0$ of $\Lambda$ in $\hG/\bG$, which is also a subgroup of $\hG/\bG$ of finite index; replace the virtual section $(s,\Lambda)$ with $(s|_{\Lambda_0},\Lambda_0)$. In this manner, we may assume
that $\Lambda$ is a \emph{normal} subgroup of $\hG/\bG$ of finite index.
By Lemma \ref{lem cl equivalence 3}, $\cl_{\hG,\bG}=\cl_{\hG,\ppi^{-1}(\Lambda)}$.
Since $G/{\ppi}^{-1}(\Lambda)$ is a finite group, by Theorem \ref{thm cl equivalence 2},
 $\cl_{\hG,\ppi^{-1}(\Lambda)}$ and $\cl_G$ are equivalent on $[\hG,\bG]$.
\end{proof}

\begin{proof}[Proof of Proposition~$\ref{pure braid and commutator}$]
Set $\hG=B_n$ and $\bG=P_n \cap [B_n,B_n]$. Recall that $\hG/\bG \cong \ZZ \times \mathfrak{S}_n$.
It is known that $Z(B_n) \cong \langle \Delta^2 \rangle$, where $\Delta$ is the half twist, more precisely,
\[\Delta=(\sigma_1\sigma_2\cdots\sigma_{n-1} )(\sigma_1\sigma_2\cdots\sigma_{n-2})\cdots (\sigma_1\sigma_2) \sigma_1.\]
Set
\[\Lambda = \ppi(Z(B_n)) \cong \big( n(n-1) \ZZ \big) \times \{ 0 \} \subset \ZZ \times \mathfrak{S}_n.\]
Note that $\Lambda$ is a normal subgroup of $\hG/\bG$ of finite index.
Define a homomorphism $s \colon \Lambda \cong \ZZ \to \ppi^{-1}(\Lambda)$ by $s(k) = \Delta^{2k}$.
Then $(s,\Lambda)$ is a virtual section of $\ppi$ and $s(\Lambda) \subset Z(G)$.
Therefore, by Theorem \ref{thm cl equivalence 3}, $\cl_{\hG}$ and $\cl_{\hG,\bG}$ are equivalent on $[\hG,\bG]$.
\end{proof}


The authors do not know whether $\cl_{B_n}$ and $\cl_{B_n,[P_n,P_n]}$ are equivalent on $[B_n,[P_n,P_n]]$.

As the second application, we exhibit examples related to free products. Let $A$, $B$ be two groups (possibly non-commutative or infinite). Let $\ppi\colon A\ast B\to A\times B$ be the canonical surjection. Let $\hG=A\ast B$ and $\bG=\mathrm{Ker}(\ppi)$. Then, it is well known that $\bG$ is a free group with basis $\{[a,b]\; | \; a\in A- \{e_A\},\ b\in B- \{e_B\}\}$.
In the setting above, we have the following.

\begin{prop}\label{prop=freeindex}
Let $A^{\mathrm{ab}}$ and $B^{\mathrm{ab}}$ be the abelianizations of $A$ and $B$, respectively. Regard them as $\mathbb{Z}$-modules. Then we have
\[
\bG/[\hG,\bG]\cong A^{\mathrm{ab}}\otimes_{\mathbb{Z}} B^{\mathrm{ab}}
\]
as $\mathbb{Z}$-modules.
\end{prop}

\begin{prop}\label{prop=freescl}
Assume that $A$ has an involutive element, that means, there exists $z\in A- \{e_A\}$ with $z^2=e_A$. Assume also that $\#B\geq 2$ and $(\# A,\#B)\ne (2,2)$. Then, on $[\bG,\bG]$, $\mathrm{scl}_{\hG,\bG}$ is \emph{not} equivalent to $\mathrm{scl}_{\bG}$.
\end{prop}

Before proceeding to the proofs, we make the following two remarks. First, if $\#A<\infty$ and if $\#B<\infty$, then $[\hG:\bG]=(\#A)\cdot (\#B)<\infty$. Note that $\bG/([\hG,\hG]\cap \bG)$ always equals $0$. Hence, under the assumptions above, Proposition~\ref{prop=freeindex} implies that
\[
\#(([\hG,\hG]\cap \bG)/[\hG,\bG])=\#(A^{\mathrm{ab}}\otimes_{\mathbb{Z}} B^{\mathrm{ab}}).
\]
This number can be much larger than $[\hG:\bG]=(\#A)\cdot (\#B)$. In particular, in Proposition~\ref{prop 8.5},  we can\emph{not} bound $\#([\hG,\hG]/[\hG,\bG])$ from above by any polynomial function in $[\hG:\bG]$.

Secondly, if either $A$ or $B$ is finite, then $1\to \bG \to \hG \to Q=A\times B\to 1$ virtually splits. Hence, Theorem~\ref{scl eq} applies to this case. Nevertheless, Proposition~\ref{prop=freescl} implies that in several cases, the restriction of $\mathrm{scl}_{\hG,\bG}$ on $[\bG,\bG]$ is not equivalent to $\mathrm{scl}_{\bG}$. We remark that the Cauchy theorem and the Feit--Thompson theorem imply the following: if a finite group does not admit an involutive element, then it must be solvable.

In the proofs of Propositions~\ref{prop=freeindex} and \ref{prop=freescl}, we use the symbol $s_{a,b}:=[a,b]\in \bG$ for $a\in A$ and $b\in B$.

\begin{proof}[Proof of Proposition~$\ref{prop=freeindex}$]
For $a\in A$ and $b\in B$, let $\overline{s}_{a,b}$ be the image of $s_{a,b}$ in $\bG/[\hG,\bG]$. Since, $[\hG,\bG]\supseteq [\bG,\bG]$, we consider $\bG/[\hG,\bG]$ as an additive group. Let $c\in A$ and $d\in B$. Then direct computation shows that
\[
c \cdot s_{a,b}\cdot c^{-1}=s_{c a,b}\cdot s_{c,b}^{-1}\quad \textrm{and}\quad d \cdot s_{a,b} \cdot d^{-1}=s_{a,d}^{-1}\cdot s_{a,d b}.
\]
Therefore, $\bG/[\hG,\bG]$ is the additive group generated by $\{\overline{s}_{a,b}\; | \; a\in A,\ b\in B\}$ subject to the relations
\[
\overline{s}_{c a,b}=\overline{s}_{c,b}+\overline{s}_{a,b} \quad \textrm{and}\quad
\overline{s}_{a,d b}=\overline{s}_{a,d}+\overline{s}_{a,b}
\]
for $c\in A$ and $d \in B$. These relations imply that $\overline{s}_{a,b}$ only depends on the image of $(a,b)$ under the natural projection $A\times B \twoheadrightarrow A^{\mathrm{ab}}\oplus B^{\mathrm{ab}}$. Hence, for $\alpha\in  A^{\mathrm{ab}}$ and $\beta\in B^{\mathrm{ab}}$, the element $\overline{s}(\alpha,\beta):=\overline{s}_{a,b}\in \bG/[\hG,\bG]$ is well-defined. Here, $(a,b)$ is a lift of $(\alpha,\beta)$ for the map $A\times B \twoheadrightarrow A^{\mathrm{ab}}\oplus B^{\mathrm{ab}}$.

To summerize, $\bG/[\hG,\bG]$ is the $\mathbb{Z}$-module generated by $\{\overline{s}(\alpha,\beta)\; | \; \alpha\in A^{\mathrm{ab}},\ \beta\in B^{\mathrm{ab}}\}$ subject to the relations
\[
\overline{s}(\alpha_1+\alpha_2,\beta)=\overline{s}(\alpha_1,\beta)+\overline{s}(\alpha_2,b) \quad \textrm{and}\quad
\overline{s}(\alpha,\beta_1+\beta_2)=\overline{s}(\alpha,\beta_1)+\overline{s}(\alpha,\beta_2)
\]
for $\alpha,\alpha_1,\alpha_2\in A^{\mathrm{ab}}$ and $\beta,\beta_1,\beta_2 \in B^{\mathrm{ab}}$. This exactly means that $\bG/[\hG,\bG]\cong A^{\mathrm{ab}}\otimes_{\mathbb{Z}} B^{\mathrm{ab}}$, as desired.
\end{proof}

\begin{proof}[Proof of Proposition~$\ref{prop=freescl}$]
Let $z\in A$ be an involutive element; hence $z^{-1}=z$. We claim that there exists $h\in \bG$ such that $h$ and $zhz=zhz^{-1}$ do not commute. Indeed, if $\#B\geq 3$, then take two distinct elements $b_1,b_2\in B- \{e_B\}$ and set $h=s_{z,b_1}\cdot s_{z,b_2}$. Otherwise,  $\#A> 2$ and $B\ne \{e_B\}$ by assumption; in this case, take $a\in A- \{e_A,z\}$ and $b\in B- \{e_B\}$, and set $h=s_{a,b}$. Let
\[
x=h(zhz)h^{-1}(zh^{-1}z) \quad \in [\bG,\bG].
\]
Note that $zxz=x^{-1}$. This implies that $\mathrm{scl}_{\hG,\bG}(x)=0$ by Theorem~\ref{main thm}. Indeed, note that $f(x)=f(zxz)=-f(x)$ for every $f\in \QQQ^{\mathrm{h}}(\bG)^{\hG}$. (Or, observe that $x^{2n}=(zx^{-n}z)x^n=[z,x^{-n}]$ for every $n\in \ZZ$.) On the other hand, $x$ is a non-trivial element of the commutator subgroup of the free group $\bG$. Since the rank of $\bG$ is at least $2$ by assumption, we conclude that $\mathrm{scl}_{\bG}(x)>0$. Indeed, consider an appropriate counting quasimorphism on $\bG$ constructed by Brooks \cite{Brooks}, and apply Theorem~\ref{thm Bavard}. It completes our proof.
\end{proof}

\bibliographystyle{amsalpha}
\bibliography{Bavard_9.bib}

\newcommand{\etalchar}[1]{$^{#1}$}
\providecommand{\bysame}{\leavevmode\hbox to3em{\hrulefill}\thinspace}
\providecommand{\MR}{\relax\ifhmode\unskip\space\fi MR }
\providecommand{\MRhref}[2]{%
  \href{http://www.ams.org/mathscinet-getitem?mr=#1}{#2}
}
\providecommand{\href}[2]{#2}
\begin{thebibliography}{KKMM21}

\bibitem[Bav91]{Bavard}
Christophe Bavard, \emph{Longueur stable des commutateurs}, Enseign. Math. (2)
  \textbf{37} (1991), no.~1-2, 109--150.

\bibitem[BIP08]{BIP}
Dmitri Burago, Sergei Ivanov, and Leonid Polterovich,
  \emph{Conjugation-invariant norms on groups of geometric origin}, Groups of
  diffeomorphisms, Adv. Stud. Pure Math., vol.~52, Math. Soc. Japan, Tokyo,
  2008, pp.~221--250.

\bibitem[BM19]{BM}
Michael Brandenbursky and Micha{\l} Marcinkowski, \emph{Aut-invariant norms and
  {A}ut-invariant quasimorphisms on free and surface groups}, Comment. Math.
  Helv. \textbf{94} (2019), no.~4, 661--687.

\bibitem[Bro81]{Brooks}
Robert Brooks, \emph{Some remarks on bounded cohomology}, Riemann surfaces and
  related topics: {P}roceedings of the 1978 {S}tony {B}rook {C}onference
  ({S}tate {U}niv. {N}ew {Y}ork, {S}tony {B}rook, {N}.{Y}., 1978), Ann. of
  Math. Stud., vol.~97, Princeton Univ. Press, Princeton, N.J., 1981,
  pp.~53--63.

\bibitem[Cal09]{Ca}
Danny Calegari, \emph{scl}, MSJ Memoirs, vol.~20, Mathematical Society of
  Japan, Tokyo, 2009.

\bibitem[CMS14]{CMS}
Danny Calegari, Naoyuki Monden, and Masatoshi Sato, \emph{On stable commutator
  length in hyperelliptic mapping class groups}, Pacific J. Math. \textbf{272}
  (2014), no.~2, 323--351.

\bibitem[CZ11]{CZ}
Danny Calegari and Dongping Zhuang, \emph{Stable {$W$}-length}, Topology and
  geometry in dimension three, Contemp. Math., vol. 560, Amer. Math. Soc.,
  Providence, RI, 2011, pp.~145--169.

\bibitem[EK01]{EK}
H.~Endo and D.~Kotschick, \emph{Bounded cohomology and non-uniform perfection
  of mapping class groups}, Invent. Math. \textbf{144} (2001), no.~1, 169--175.

\bibitem[EP03]{EP03}
Michael Entov and Leonid Polterovich, \emph{Calabi quasimorphism and quantum
  homology}, Int. Math. Res. Not. (2003), no.~30, 1635--1676.

\bibitem[EP06]{EP06}
\bysame, \emph{Quasi-states and symplectic intersections}, Comment. Math. Helv.
  \textbf{81} (2006), no.~1, 75--99.

\bibitem[Fri17]{Fr}
Roberto Frigerio, \emph{Bounded cohomology of discrete groups}, Mathematical
  Surveys and Monographs, vol. 227, American Mathematical Society, Providence,
  RI, 2017.

\bibitem[GS87]{GS}
\'{E}tienne Ghys and Vlad Sergiescu, \emph{Sur un groupe remarquable de
  diff\'{e}omorphismes du cercle}, Comment. Math. Helv. \textbf{62} (1987),
  no.~2, 185--239.

\bibitem[Hat02]{H}
Allen Hatcher, \emph{Algebraic topology}, Cambridge University Press,
  Cambridge, 2002.

\bibitem[Ish14]{I}
Tomohiko Ishida, \emph{Quasi-morphisms on the group of area-preserving
  diffeomorphisms of the 2-disk via braid groups}, Proc. Amer. Math. Soc. Ser.
  B \textbf{1} (2014), 43--51.

\bibitem[Kaw16]{Ka16}
Morimichi Kawasaki, \emph{Relative quasimorphisms and stably unbounded norms on
  the group of symplectomorphisms of the {E}uclidean spaces}, J. Symplectic
  Geom. \textbf{14} (2016), no.~1, 297--304.

\bibitem[Kaw17]{Ka17}
\bysame, \emph{Bavard's duality theorem on conjugation-invariant norms},
  Pacific J. Math. \textbf{288} (2017), no.~1, 157--170.

\bibitem[Kaw18]{Ka18}
\bysame, \emph{Extension problem of subset-controlled quasimorphisms}, Proc.
  Amer. Math. Soc. Ser. B \textbf{5} (2018), 1--5.

\bibitem[Kim18]{Ki}
Mitsuaki Kimura, \emph{Conjugation-invariant norms on the commutator subgroup
  of the infinite braid group}, J. Topol. Anal. \textbf{10} (2018), no.~2,
  471--476.

\bibitem[KK19]{KK}
Morimichi Kawasaki and Mitsuaki Kimura, \emph{$\hat{G}$-invariant
  quasimorphisms and symplectic geometry of surfaces}, to appear in
  \emph{Israel J.\ Math}, arXiv:1911.10855v2 (2019).

\bibitem[KKM{\etalchar{+}}21]{KKMMM}
Morimichi Kawasaki, Mitsuaki Kimura, Shuhei Maruyama, Takahiro Matsushita, and
  Masato Mimura, \emph{The space of non-extendable quasimorphisms}, preprint,
  arXiv:2107.08571v3 (2021).

\bibitem[KKMM21]{KKMM}
Morimichi Kawasaki, Mitsuaki Kimura, Takahiro Matsushita, and Masato Mimura,
  \emph{Commuting symplectomorphisms on a surface and the flux homomorphism},
  preprint, arXiv:2102.12161v1 (2021).

\bibitem[Mar21]{Maruyama}
Shuhei Maruyama, \emph{A note on stable commutator length in braided
  {P}tolemy-{T}hompson groups}, Kodai Math. J. 44 (2021), no.2, 317-322 (20F12)
  (2021).

\bibitem[Mim10]{Mimura}
Masato Mimura, \emph{On quasi-homomorphisms and commutators in the special
  linear group over a {E}uclidean ring}, Int. Math. Res. Not. IMRN (2010),
  no.~18, 3519--3529.

\bibitem[Py06]{Py06}
Pierre Py, \emph{Quasi-morphismes et invariant de {C}alabi}, Ann. Sci.
  \'{E}cole Norm. Sup. (4) \textbf{39} (2006), no.~1, 177--195.

\bibitem[She14]{Shelukhin}
Egor Shelukhin, \emph{The action homomorphism, quasimorphisms and moment maps
  on the space of compatible almost complex structures}, Comment. Math. Helv.
  \textbf{89} (2014), no.~1, 69--123. \MR{3177909}

\bibitem[Sht16]{Sh}
A.~I. Shtern, \emph{Extension of pseudocharacters from normal subgroups,
  {III}}, Proc. Jangjeon Math. Soc. \textbf{19} (2016), no.~4, 609--614.

\bibitem[Zhu08]{Zhuang}
Dongping Zhuang, \emph{Irrational stable commutator length in finitely
  presented groups}, J. Mod. Dyn. \textbf{2} (2008), no.~3, 499--507.
  \MR{2417483}

\end{thebibliography}

\end{document}